\documentclass[11pt,a4paper,twoside,reqno]{amsart}
\usepackage[utf8]{inputenc}
\usepackage[margin=0.9in]{geometry} 
\usepackage{amsthm}
\usepackage{amsmath}
\usepackage{amsfonts}
\usepackage{amssymb}
\usepackage{chngcntr}
\usepackage{url}
\usepackage{hyperref}
\usepackage[makeroom]{cancel}
\usepackage{enumerate}
\usepackage{mathrsfs}
\usepackage{xcolor}
\usepackage{dsfont}
\usepackage{mathtools}
\mathtoolsset{showonlyrefs}
\newtheorem{Theorem}{Theorem}

\newtheorem{Proposition}[Theorem]{Proposition}
\newtheorem{Corollary}[Theorem]{Corollary}
\newtheorem{Lemma}[Theorem]{Lemma}

\newtheorem{Definition}[Theorem]{Definition}

\newcommand{\supp}{\operatorname{supp}}

\newcommand{\Var}{\operatorname{Var}}
\newcommand{\eps}{\varepsilon}
\newcommand{\ed}[1]{\textcolor{black}{#1}}

\counterwithin{Theorem}{section}
\numberwithin{equation}{section}

\allowdisplaybreaks

\title{Proving the Duffin--Schaeffer conjecture without GCD graphs}
\author{Manuel Hauke-Treuer, Santiago Vazquez, and Aled Walker}
\begin{document}
\begin{abstract}
We present a novel proof of the Duffin--Schaeffer conjecture in metric Diophantine approximation. Our proof is heavily motivated by the ideas of Koukoulopoulos--Maynard's breakthrough first argument, but simplifies and strengthens several technical aspects. In particular, we avoid any direct handling of GCD graphs and their `quality'. We also consider the metric quantitative theory of Diophantine approximations, improving the $(\log \Psi(N))^{-C}$ error-term of Aistleitner, Borda and the first named author to $\exp(-(\log \Psi(N))^{\frac{1}{2} - \varepsilon})$. 
\end{abstract}
\maketitle
\tableofcontents

\author{Manuel Hauke-Treuer}
{\footnotesize 

University of York

Department of Mathematics

YO10 5DD York, United Kingdom

Current Address: Graz University of Technology,

8010 Graz, Austria

Email: \texttt{hauke@math.tugraz.at}}
\vspace{5mm}

\author{Santiago Vazquez}
{\footnotesize

King's College London

Department of Mathematics, Strand Building, Strand

WC2R 2LS London, United Kingdom

Email: \texttt{santiago.vazquez\_saez@kcl.ac.uk}

}
\vspace{5mm}

\author{Aled Walker}
{\footnotesize

King's College London

Department of Mathematics, Strand Building, Strand

WC2R 2LS London, United Kingdom

Email: \texttt{aledwalker@gmail.com}

}
\vspace{5mm}

\section{Introduction}
The Duffin--Schaeffer conjecture was a long-standing open problem in the area of Diophantine approximation. It was concerned with the following broad question: given an arbitrary fixed function $\psi: \mathbb{N} \to \mathbb{R}_{\geqslant 0}$, and a typical irrational number $\alpha \in [0,1]$, when should one expect to find infinitely many rational numbers $a/n$ with 
\begin{equation}
\label{eq:approximationequation}
\Big\vert \alpha - \frac{a}{n} \Big\vert \leqslant \frac{\psi(n)}{n}?
\end{equation}
If $\psi$ is reasonably `well-behaved' (such as $\psi(n) = 1/n$, or $\psi(n) = 1/(n \log n)$, or in general a function with $n\psi(n)$ monotonically decreasing) then a satisfactory answer was already known by Khintchine \cite{K24} in 1924. But if $\psi$ is a wilder function, in particular if $\psi(n) = 0$ for all but a very sparse set of denominators $n$, then it is much more challenging to understand whether inequality \eqref{eq:approximationequation} is satisfied infinitely often.

In 1941, Duffin and Schaeffer \cite{DS41} proposed a criterion on the approximation function $\psi$ which should determine whether such approximations exist. 
After certain partial progress, by Duffin and Schaeffer themselves, Gallagher \cite{G61}, Erd\H{os} \cite{Er70}, Vaaler\cite{Va78}, Pollington and Vaughan \cite{PV90}, and more recently in \cite{Aistleitner_2014,Aistleitner_Lachmann_Munsch_Technau_Zafeiropoulos_2019,Beresnevich_Harman_Haynes_Velani_2013,Haynes_Pollington_Velani_2012}, the Duffin--Schaeffer conjecture was finally resolved by Koukoulopoulos and Maynard \cite{KM19} in July 2019. 

\begin{Theorem}[Theorem 1, \cite{KM19}]
\label{Theorem:KM}
Let $\psi:\mathbb{N} \longrightarrow \mathbb{R}_{\geqslant 0}$ be a function satisfying
\begin{equation}
\label{eq:DScriterion}
\sum\limits_{n=1}^\infty \frac{ \psi(n) \varphi(n)}{n} = \infty,
\end{equation} where $\varphi$ denotes the Euler totient function. Then, for almost all $\alpha \in [0,1]$, inequality \eqref{eq:approximationequation} has infinitely many coprime integer solutions $(a,n)$. 
\end{Theorem}
\noindent By refining and developing the Koukoulopoulos--Maynard method, Aistleitner, Borda and the first named author were able to count solutions to \eqref{eq:approximationequation}. 
\begin{Theorem}[Theorem 1, \cite{ABH23}]
\label{Theorem:ABH}
Let $\psi: \mathbb{N} \to [0,1/2]$ be a function satisfying \eqref{eq:DScriterion}. Write $S(N,\alpha)$ for the number of coprime integer solutions $(a,n)$ to the inequality \eqref{eq:approximationequation} subject to $1 \leqslant n \leqslant N$, and let \begin{equation}
    \label{eq:def_Psi}
 \Psi(N): = \sum_{n = 1}^{N} \frac{2 \varphi(n) \psi(n)}{n}.\end{equation} Then, for all $A>0$ and almost all $\alpha \in [0,1]$ \[ S(N, \alpha) = \Psi(N)\left( 1 + O _{\alpha,A}\Big( \frac{1}{(\log \Psi(N))^A}\Big)\right)\] as $N \to \infty$. In particular, for almost all $\alpha \in [0,1]$ one has $S(N,\alpha) \to \infty$ if $\Psi(N) \to \infty$. 
\end{Theorem}

These results may be motivated as follows. Assuming that $\psi$ takes values in $[0,1/2]$, to avoid some minor degeneracy issues (Theorem \ref{Theorem:KM} for the case where $\psi(n) \geqslant \tfrac{1}{2}$ for infinitely many $n$ was resolved in \cite{PV90}), define 
\begin{equation}
\label{eq:Aq}
 \mathcal{A}_n: = \bigcup\limits_{\substack{0 \leqslant a \leqslant n \\ \gcd(a,n) = 1}} \Big[ \frac{a}{n} - \frac{\psi(n)}{n}, \frac{a}{n} + \frac{\psi(n)}{n}\Big] \cap [0,1].
\end{equation} Then, writing $\lambda$ for the Lebesgue measure, we have \[ \lambda(\mathcal{A}_n) = \frac{2 \varphi(n) \psi(n)}{n}.\] Viewing $\alpha$ as a random variable chosen uniformly at random from $[0,1]$, one has $S(N, \alpha) = \sum_{n \leqslant N} \mathds{1}_{[\alpha \in \mathcal{A}_n]}$ and $\Psi(N) = \sum_{n \leqslant N} \lambda(\mathcal{A}_n) = \mathbb{E}_\alpha S(N, \alpha)$. Theorem \ref{Theorem:ABH} is therefore the statement that the random variable $S(N, \alpha)$ concentrates around its mean. For more background on the motivation behind the Duffin--Schaeffer conjecture, we direct the reader to the introductions of \cite{ABH23,KM19}.\\ 

In this paper, we revisit the proofs of Theorems \ref{Theorem:KM} and \ref{Theorem:ABH}. It transpires that the core ideas of \cite{ABH23, KM19} may be `repackaged' and, when combined with some new arguments, the overall proofs may be substantially shortened and conceptually simplified. \ed{Furthermore, we note that since the appearance of the first version of this manuscript, the ideas and tools developed in this article have been successfully applied by Koukoulopoulos, Lamzouri and Lichtman \cite{KLL25} to solve a long-standing open Erd\H{o}s problem about integer dilation approximation.}\\

Though we consider our main novel contribution to be the statement and proof of the technical estimate Theorem \ref{Theorem:maintheorem} below, as a corollary we also reduce the size of the error term in Theorem \ref{Theorem:ABH}.
\begin{Corollary}[Improved error term for quantitative metric approximations]
\label{Corollary:quantitativemetricapproximations}
Under the same hypothesis as Theorem \ref{Theorem:ABH}, for all $\varepsilon >0$ and almost every $\alpha \in [0,1]$,
 \[ S(N,\alpha) = \Psi(N)\left(1 + O_{\varepsilon,\alpha}\left(\exp\big(-(\log \Psi(N))^{1/2 -\varepsilon}\big)\right)\right),
    \] 
\end{Corollary}
\noindent Moreover, as we will explain below, our Theorem \ref{Theorem:maintheorem} is potentially strong enough to prove a near square-root cancellation error term $O_{\varepsilon}(\Psi(N)^{-\frac{1}{2} + \varepsilon})$ in Corollary \ref{Corollary:quantitativemetricapproximations}. 
Up to the $\Psi(N)^{\varepsilon}$ term, the error term is in accordance with the one known for 
the Khintchine setup, that is, considering monotonic $\psi$ and not only counting coprime approximations, see \cite{S60} or \cite[Chapter 4]{H98}. Furthermore we note that due to the diagonal contribution, a square-root error term is the best that can be obtained by an $L^2$-approach; to the best of the authors' knowledge, all proofs in this area rely on this tool.

However, any significant improvement on Corollary 1.3 would also require development to the sieve-theoretic estimates of [2, Lemma 5], on which our argument relies. As it happens, an independent work of Koukoulopoulos--Maynard--Yang \cite{KMY24}, elaborated at the same time as our present manuscript, established such an almost square-root cancellation, by improving on precisely these sieve elements. They also prove a version of our main technical estimates (Theorem 1.7 below). However, their arguments use the complicated technical formalism of GCD graphs, which we deliberately avoid.\\

Theorems \ref{Theorem:KM} and \ref{Theorem:ABH} were proved by bounding $\Var_{\alpha}(S(N,\alpha))$ above. As \[\Var_{\alpha}(S(N,\alpha)) = \sum_{n,m \leqslant N} \lambda(\mathcal{A}_n \cap \mathcal{A}_m) - \Psi(N)^2,\] one requires a so-called `overlap estimate' on $\lambda(\mathcal{A}_n \cap \mathcal{A}_m)$. Utilising a $0$-$1$ law of Gallagher \cite{G61}, Koukoulopoulos--Maynard were satisfied with proving an upper bound $\sum_{n,m \leqslant N} \lambda(\mathcal{A}_n \cap \mathcal{A}_m) \ll \Psi(N)^2$ that might lose a constant factor in the implied constant. It was enough therefore to use an older overlap estimate of Pollington--Vaughan on $\lambda(\mathcal{A}_n \cap \mathcal{A}_m$) (\cite[Lemma 5.3]{KM19} and \cite[pp 195-196]{PV90}), which also lost such a factor. However, as Aistleitner, Borda and the first named author needed to prove $\Var_{\alpha}(S(N, \alpha)) = o_{N \to \infty}(\Psi(N)^2)$, they required an asymptotic on $\sum_{n,m \leqslant N} \lambda(\mathcal{A}_n \cap \mathcal{A}_m)$. They thus required an upper bound of the form $\lambda(\mathcal{A}_n \cap \mathcal{A}_m) \leqslant \lambda(\mathcal{A}_n) \lambda(\mathcal{A}_m)(1 + E(n,m))$, with $E(n,m)$ an error that could be shown to be $o(1)$ on average over $n$ and $m$. The Pollington--Vaughan overlap estimate was insufficient. 

Here we state the refined overlap estimate from \cite{ABH23}, along with a corollary we will use later. Note the breakdown of independence of the events $\alpha \in \mathcal{A}_n$ and $\alpha \in \mathcal{A}_m$ when $n$ has many prime factors that do not divide $m$ (or vice versa), or when $\gcd(n,m)$ is large. 

\begin{Lemma}[Overlap estimate]
\label{Lemma:overlapestimate}
	For any positive integers $n \neq m$ and any reals $u \geqslant 1$ and $T \geqslant 2$, we have
	\begin{equation} \label{eq:overlapclaim}
	\lambda (\mathcal{A}_n \cap \mathcal{A}_m) \leqslant \lambda (\mathcal{A}_n) \lambda (\mathcal{A}_m) \left( 1 + O \left( u^{-u/2} + \frac{T^u \log (D+2) \log T}{D} \right) \right) \prod_{\substack{p \mid \frac{nm}{\gcd(n,m)^2}, \\ p>T}} \left( 1+ \frac{1}{p-1} \right)
	\end{equation}
	with an absolute implied constant, where 
 \begin{equation*}
D = D(n,m):= \frac{\max \left(  m \psi(n), n \psi(m) \right)}{\gcd (n,m)}.
\end{equation*} In particular, writing
$F_{\rho}(x) = \exp((\log x)^{1/2-\rho})$, we have for any $0 < \rho < 1/2$ that
	\begin{equation}
\label{eq:overlap_optimized}
	\lambda (\mathcal{A}_n\cap \mathcal{A}_m) \leqslant \lambda (\mathcal{A}_n) \lambda (\mathcal{A}_m) \big( 1 + O_{\rho}\left(F_{\rho}(D)^{-1}\right) \big) \prod_{\substack{p \mid \frac{nm}{\gcd(n,m)^2}, \\ p>F_{\rho}(D)}} \left( 1+ \frac{1}{p-1} \right)
	\end{equation}
	with an implied constant depending only on $\rho$.
 Furthermore, we have 
 \begin{equation}
     \label{eq:u_const_bound}
     \lambda(\mathcal{A}_n\cap \mathcal{A}_m) \ll \lambda (\mathcal{A}_n) \lambda (\mathcal{A}_m) \prod_{\substack{p \mid \frac{nm}{\gcd(n,m)^2}, \\ p>D}} \left( 1+ \frac{1}{p} \right).
 \end{equation}
\end{Lemma}

\begin{proof}
    Equation \eqref{eq:overlapclaim} is precisely \cite[Lemma 5]{ABH23}, which has a short three-page proof combining the classical overlap approach of Pollington--Vaughan with the Fundamental Lemma of sieve theory. Equation \eqref{eq:overlap_optimized} then follows from \eqref{eq:overlapclaim} by choosing 
    $u = \sqrt{\log D}$ and $T = F_{\rho}(D)$. Equation \eqref{eq:u_const_bound} follows from \eqref{eq:overlapclaim} by choosing $u= 2$, $T = D^{1/10}$ and applying Mertens' Theorem in the form of $\prod_{D^{1/10} < p < D}(1 + \frac{1}{p-1}) \ll 1$. (In fact equation \eqref{eq:u_const_bound} would also follow from Pollington--Vaughan, in \cite[Lemma 5.3]{KM19} or \cite[pp. 195-196]{PV90}). 
\end{proof}

Using \eqref{eq:overlapclaim}, one seeks a bound of the form 
\begin{equation}
\label{eq:twosidedbound}
 \sum_{n,m \leqslant N} \lambda(\mathcal{A}_n\cap \mathcal{A}_m) \leqslant \Psi(N)^2(1 + o_{N \to \infty}(1)),
\end{equation} for some quantitative $o_{N \to \infty}(1)$ term. After this the estimate $\Var_{\alpha}(S(N,\alpha)) = o_{N \to \infty}(\Psi(N)^2)$ would follow immediately, and Theorem \ref{Theorem:ABH} would be proved by a standard procedure that combines Chebyshev's inequality and the first Borel-Cantelli Lemma; this allows moving, at a minimal cost in the speed of convergence, from a variance bound to almost sure asymptotics, provided the quantitative decay is sufficiently rapid, see e.g. \cite[Chapter 4]{H98}. In the case of Theorem \ref{Theorem:ABH}, this is handled on pp. 210-211 of \cite{ABH23}, and we also make use of this method in Section \ref{Section:quantitativeapplication} to obtain Corollary \ref{Corollary:quantitativemetricapproximations}. 

 We are required to show that the contribution to $\sum_{n,m \leqslant N} \lambda(\mathcal{A}_n\cap \mathcal{A}_m)$ coming from pairs $(n,m)$ with small $D$ (i.e.\! pairs for which $\gcd(n,m)$ is large) is negligible, when considered simultaneously with the potential contribution from pairs $(n,m)$ for which \[\prod\limits_{\substack{p \vert \frac{nm}{\gcd(n,m)^2}\\ p > T}} \Big( 1 + \frac{1}{p-1}\Big)\] is much larger than $1$. This is the content of \cite[Proposition 5.4]{KM19}, Koukoulopoulos--Maynard's breakthrough estimate. The matter is also considered in \cite[Propositions 6 and 7]{ABH23}, whose proofs build on the Koukoulopoulos--Maynard method with refined estimates. Our main technical result is a general bilinear bound whose short self-contained proof implies all of these results. 

To prepare the statement of this bilinear bound, we introduce some notation for the weight associated to certain sets of integers and pairs of integers. 
\begin{Definition}[Measures]
Let $\psi,\theta: \mathbb{N} \longrightarrow \mathbb{R}_{ \geqslant 0}$ be finitely supported, and let $f,g:\mathbb{N} \longrightarrow \mathbb{R}_{\geqslant 0}$ be multiplicative functions. For $v \in \mathbb{N}$ we define \[\mu_{\psi}^f(v): = \frac{f(v) \psi(v)}{v}.\] When $V \subset \mathbb{N}$ we define \[ \mu_{\psi}^f(V) := \sum_{v \in V} \mu_{\psi}^f(v) = \sum_{v \in V} \frac{f(v) \psi(v)}{v}.\] If $\mathcal{E} \subset \mathbb{N} \times \mathbb{N}$, we define \[ \mu_{\psi,\theta}^{f,g}(\mathcal{E}): = \sum_{(v,w) \in \mathcal{E}} \mu_{\psi}^f(v) \mu_{\theta}^g(w).\] 
\end{Definition}
\noindent We now define a general set of pairs of integers $(v,w)$ for which $D$ is small and  many large primes divide $vw/\gcd(v,w)^2$. 
\begin{Definition}
\label{Definition:edge_set}
Let $\psi,\theta: \mathbb{N} \longrightarrow \mathbb{R}_{ \geqslant 0}$ be finitely supported, with $V_{\psi} = \supp \psi$ and $W_{\theta} = \supp \theta$, and let $t \geqslant 1$ and $C \in \mathbb{R}$. Then we define \[ \mathcal{E}_{\psi,\theta}^{t, C}: = \Big\{(v,w) \in V_{\psi} \times W_{\theta}: \, D_{\psi,\theta}(v,w) \leqslant 1, \, \sum_{\substack{p \vert \frac{vw}{\gcd(v,w)^2} \\ p \geqslant t}} \frac{1}{p}\geqslant C\Big\},\] where \[ D_{\psi,\theta}(v,w) = \frac{\max(w\psi(v), v \theta(w))}{\gcd(v,w)}.\] 
\end{Definition}

\noindent We can now state our main result. 

\begin{Theorem}[Main technical result]
\label{Theorem:maintheorem}
Let $\varepsilon \in (0,2/5]$. Then there exists $p_0(\varepsilon) >0$ such that the following holds. Let $\psi,\theta: \mathbb{N} \longrightarrow \mathbb{R}_{\geqslant 0}$ be finitely supported, $V_{\psi} = \supp \psi$, $W_{\theta} = \supp \theta$ and \[ \mathcal{P}_{\psi,\theta}: = \{p: \, \exists \, (v,w) \in V_{\psi}\times W_{\theta} \, \text{s.t.} \, p \vert vw\}.\] Let $P_{\psi,\theta}(\varepsilon): = p_0(\varepsilon) + \vert \mathcal{P}_{\psi,\theta} \cap [1, p_0(\varepsilon)]\vert$. Let $f,g:\mathbb{N} \longrightarrow \mathbb{R}_{\geqslant 0}$ be multiplicative functions for which 
\begin{equation}\label{eq:conv_bound}(1 \star f)(n) \leqslant n, \qquad \text{and} \qquad (1 \star g)(n) \leqslant n \qquad  (\text{for all }n \geqslant 1).\end{equation} 
Suppose that $\mathcal{E} \subset \mathcal{E}_{\psi, \theta}^{t, C} \cap (V \times W)$.
Then for all $t \geqslant 1$ and any $C \in \mathbb{R}$ we have
\begin{equation}
\label{eq:maintheoremconclusion}
\mu_{\psi,\theta}^{f,g}(\mathcal{E}) \leqslant 1000^{P_{\psi,\theta}(\varepsilon)}( \mu_{\psi}^f(V) \mu_{\theta}^g(W) e^{-Ct})^{\frac{1}{2} + \varepsilon}.
\end{equation}
\end{Theorem}

\subsection*{Remarks}
\begin{enumerate}
\item The conditions of the theorem are satisfied if $f = g = \varphi$, in which case $\mu_{\psi}^f(V_{\psi}) =  \frac{\Psi(N)}{2}$ and the bound \eqref{eq:maintheoremconclusion} directly relates to the Duffin--Schaeffer setting. 
\item Since $P_{\psi,\theta}(\varepsilon) \leqslant 2 p_0(\varepsilon)$, we may conclude from \eqref{eq:maintheoremconclusion} that \[\mu_{\psi,\theta}^{f,g}(\mathcal{E}) \ll_{\varepsilon}\left( \mu_{\psi}^f(V) \mu_{\theta}^g(W) e^{-Ct}\right)^{\frac{1}{2} + \varepsilon},\] where the implied constant is independent of all functions and parameters except $\varepsilon$. We have stated \eqref{eq:maintheoremconclusion} in a more precise form in order to effect an induction argument later in the paper. 
\item Though Theorem \ref{Theorem:maintheorem} may only seem to apply when $D_{\psi,\theta}(v,w) \leqslant 1$, and not when\\ $D_{\psi,\theta}(v,w) \leqslant t$ or $D$ is bounded by some other parameter, we can immediately derive bounds in these more general cases, by rescaling the functions $\psi$ and $\theta$. To illustrate this, we show below that the key \cite[Proposition 5.4]{KM19} of Koukouloupolos--Maynard indeed follows immediately from Theorem \ref{Theorem:maintheorem}. In Section \ref{Section:quantitativeapplication}, we will also give the quick deductions that improve upon \cite[Propositions 6 and 7]{ABH23}, which lead eventually to Corollary \ref{Corollary:quantitativemetricapproximations}.
\item The bound \eqref{eq:maintheoremconclusion} represents near square-root cancellation from the condition on the gcds given by $D_{\psi,\theta}(v,w) \leqslant 1$, together with a further decay when pairs $(v,w)$ have many large prime factors that do not divide $v$ and $w$ to the same exponent. This suggests that Corollary \ref{Corollary:quantitativemetricapproximations} might be provable with near square-root cancellation too, but unfortunately the sieve argument inherent in the overlap estimate Lemma \ref{Lemma:overlapestimate} has too weak a relationship between the parameters $u$, $T$, and $D$ to make full use of the strength of Theorem \ref{Theorem:maintheorem}. 
\item The upper bound $2/5$ in the range of $\varepsilon$ in Theorem \ref{Theorem:maintheorem} could be replaced by any fixed value less than $1/2$, at the cost of replacing $1000$ with a larger constant.
We chose $2/5$ as the exponent in \eqref{eq:maintheoremconclusion} as $\frac{1}{2} + \frac{2}{5} = \frac{9}{10}$ corresponds to the appearance of exponents $9$ and $10$ in Koukoulopoulos--Maynard's definition of the `quality' of a GCD graph \cite[Definition 6.6(d)]{KM19}.
\end{enumerate}

\subsection*{Theorem \ref{Theorem:maintheorem} implies the Duffin--Schaeffer conjecture}

Here we show that Theorem \ref{Theorem:maintheorem} can be applied in order to prove the Duffin--Schaeffer conjecture. To be more precise, we show that \cite[Proposition 5.4]{KM19} follows immediately from  Theorem \ref{Theorem:maintheorem}. We will omit further details of how the Duffin--Schaeffer conjecture may be deduced from \cite[Proposition 5.4]{KM19}; this follows a similar (but more straightforward) path as the proof of our quantitative application in Section \ref{Section:quantitativeapplication}. One may also consult \cite[Section 5]{KM19} for the argument. 

\begin{Proposition}[Proposition 5.4 in Koukoulopoulos--Maynard \cite{KM19}]\label{Proposition:Prop5.4}
    Let $Y \geqslant X \geqslant 1$ and $\psi: \mathbb{Z} \cap [X,Y] \to [0,\infty)$ such that
    \[1 \leqslant \sum_{X \leqslant n \leqslant Y}\frac{\psi(n)\varphi(n)}{n} \leqslant 2.\]
    For each $t \geqslant 1$, denote
    \[\mathcal{E}_t := \Big\{(v,w) \in (\mathbb{Z} \cap [X,Y])^2: D(v,w) \leqslant t, \sum_{\substack{p \vert \frac{vw}{\gcd(v,w)^2} \\ p \geqslant t}} \frac{1}{p}\geqslant 10 \Big\}.\]
    Then 
    \[\sum_{(v,w) \in \mathcal{E}_t} \frac{\psi(v)\varphi(v)}{v}\frac{\psi(w)\varphi(w)}{w} \ll \frac{1}{t}.\]
\end{Proposition}

\begin{proof}[Proof of Proposition \ref{Proposition:Prop5.4} assuming Theorem \ref{Theorem:maintheorem}]
    Defining $\widetilde{\psi}(q) := \mathds{1}_{[q \in [X,Y]]}\frac{\psi(q)}{t}$, we observe that
    $\mathcal{E}_t = \mathcal{E}_{\widetilde{\psi},\widetilde{\psi}}^{t,10}$ (in the notation of Definition \ref{Definition:edge_set}). Thus applying 
    Theorem \ref{Theorem:maintheorem} with $f = g = \varphi$, we obtain
    \begin{align*}
    \sum_{(v,w) \in \mathcal{E}_t} \frac{\psi(v)\varphi(v)}{v}\frac{\psi(w)\varphi(w)}{w} = t^2 \mu_{\widetilde{\psi},\widetilde{\psi}}^{\varphi,\varphi}(\mathcal{E}_{\widetilde{\psi},\widetilde{\psi}}^{t,10}) &\ll_{\varepsilon}
    t^2 \Big(\mu_{\widetilde{\psi}}^{\varphi}([X,Y] \cap \mathbb{Z}) \mu_{\widetilde{\psi}}^{\varphi}([X,Y] \cap \mathbb{Z}) e^{-10t}\Big)^{\frac{1}{2} + \varepsilon}
    \\&\ll_{\varepsilon} t^{1 - 2\varepsilon}e^{-10t\big(\tfrac{1}{2} + \varepsilon\big)}\\ &\ll_{\varepsilon} \frac{1}{t}
\end{align*}
as required. 
\end{proof}

To finish this introduction, we indicate the structure of the remainder of the paper. In Section \ref{Section:proofstrategy}, we discuss the general proof strategy for Theorem \ref{Theorem:maintheorem} and reduce matters to proving two sub-propositions (Proposition \ref{Prop:structureofminimalcounterexample} and Proposition \ref{Prop:resolvingminimalcounterexample} below). This involves a brief discussion of the relationship between our approach and Koukoulopoulos--Maynard \cite{KM19}. The subsequent three sections are devoted to proving these propositions. In Section \ref{Section:quantitativeapplication} we deduce Corollary \ref{Corollary:quantitativemetricapproximations} from Theorem \ref{Theorem:maintheorem}.
In the appendices, we explain a relationship between our proof and the classical construction of Duffin and Schaeffer; we also make some remarks about the lemmas involving the anatomy of integers from Section \ref{Section:anatomy}. 

 \subsection*{Acknowledgements}

 This material is based upon work supported by the Swedish Research Council under grant no. 2021-06594 while the first and third authors were in residence at Institut Mittag-Leffler in Djursholm, Sweden, during the winter semester of 2024. We would like to thank the organisers of the semester programme (Pär Kurlberg, Morten Risager, and Anders Södergren). The first named author was supported by the EPSRC grant EP/X030784/1. The second named author is supported by the EPSRC CDT Grant EP/S021590/1 for the London School of Geometry and Number Theory. We are further grateful to Sam Chow, Andrew Granville, Ben Green, Dimitris Koukoulopoulos, and James Maynard, for several interesting conversations on topics related to this paper.  \ed{Further, we thank the anonymous reviewer for their careful examination of this manuscript.} 

\subsection*{Notation}
We use the standard $O$- and $o$-notations as well as the Vinogradov notations $\ll,\gg$. \ed{For a sequence of non-negative reals $x = (x_i)_{i \in \mathbb{Z}}$ and $q > 0$, we define the $q$-norms as $\Vert x\Vert_{\ell^{q^\prime}(\mathbb{Z})} = (\sum_{i \in \mathbb{Z}}x_i^{q})^{1/q}$ whenever the latter converges.}
If $\mathcal{E} \subset V \times W$ is a set of pairs and $v \in V$, then $\Gamma_{\mathcal{E}}(v): = \{w: (v,w) \in \mathcal{E}\}$ and we write $\mathcal{E}|_V = \{v \in V: \exists w \in W: (v,w) \in \mathcal{E}\}$ (with $\Gamma_{\mathcal{E}}(w), \mathcal{E}|_W$ defined analogously). For a prime $p$ we use the standard $p$-adic valuation $\nu_p$ on the rationals, and for $n \in \mathbb{N}$, we write $p^i \Vert n$ if $\nu_{p}(n) = i$. 
For multiplicative functions $f,g$, we define the Dirichlet convolution $(f \star g) (n) := \sum_{d \mid n} \ed{f(d)}g(\tfrac{n}{d})$. Throughout the paper, it is understood that if $\varepsilon \in (0,1/2)$ is given, then $q := \frac{2}{1 - 2 \varepsilon}$ and $q^\prime: = \frac{2}{1+ 2\varepsilon}$.
\\

\section{Proof strategy}
\label{Section:proofstrategy}

For the rest of the article, we fix a function $p_0: (0,1/2) \to \mathbb{R}_{\geqslant 0}$ whose values will be large enough for following arguments to hold.

Our proof of Theorem \ref{Theorem:maintheorem} splits naturally into two parts. First, we analyse the structure of potential counterexamples to the theorem, showing that any counterexample must enjoy some weak combinatorial and arithmetic structure. Second, we use the presence of such a structure to prove \eqref{eq:maintheoremconclusion} directly. This is the same strategy as used in work of Green and the third named author \cite{GW21}, which studied the gcd condition in isolation.

\begin{Proposition}[Structure of minimal potential counterexample]
\label{Prop:structureofminimalcounterexample}
Suppose Theorem \ref{Theorem:maintheorem} were false. Fix functions $\psi,\theta: \mathbb{N} \longrightarrow \mathbb{R}_{\geqslant 0}$ with finite support such that $\vert \mathcal{P}_{\psi,\theta}\vert$ is minimal over all such pairs of functions for which there exist instances of $\varepsilon$, $t$, $C$, $f$, $g$, and $\mathcal{E} \subset \mathcal{E}^{t,C}_{\psi,\theta}$ satisfying the hypotheses of Theorem \ref{Theorem:maintheorem} for which \eqref{eq:maintheoremconclusion} fails. Fix such instances of  $\varepsilon$, $t$, $C$, $f$, $g$, and $\mathcal{E}$. Write $q^\prime: = \frac{2}{1 + 2\varepsilon}$. Then there exists $\mathcal{E}^{\prime} \subset \mathcal{E}$ for which
\begin{enumerate}
\item \emph{($\mathcal{E}^{\prime}$ is a near counterexample)}: $\mu_{\psi,\theta}^{f,g}(\mathcal{E}^{\prime}) > \frac{1}{2} \cdot 1000^{P_{\psi,\theta}(\varepsilon)}( \mu_{\psi}^f(V^{\prime}) \mu_{\theta}^g(W^{\prime}) e^{-Ct})^{\frac{1}{q^\prime}}$ where
 $V^{\prime}:= \mathcal{E}^{\prime}|_V, W^{\prime}: = \mathcal{E}^{\prime}|_W$.
\item \emph{($\mathcal{E}^{\prime}$ is combinatorially structured)}: for all $v \in V^{\prime}$ and $w \in W^{\prime}$, \[\mu_{\theta}^g(\Gamma_{\mathcal{E}^{\prime}}(v)) \geqslant \frac{1}{q^\prime} \frac{\mu_{\psi,\theta}^{f,g}(\mathcal{E}^{\prime})}{\mu_{\psi}^f(V^{\prime})} \qquad \text{and} \qquad \mu_{\psi}^f(\Gamma_{\mathcal{E}^{\prime}}(w)) \geqslant \frac{1}{q^\prime} \frac{\mu_{\psi,\theta}^{f,g}(\mathcal{E}^{\prime})}{\mu_{\theta}^g(W^{\prime})} .\] 
\item \emph{($\mathcal{E}^{\prime}$ is arithmetically structured)}: there exists $N \in \mathbb{N}$ such that for all primes $p$ and for all $(v,w) \in \mathcal{E}^{\prime}$, $\vert \nu_p(v/N)\vert + \vert \nu_p(w/N)\vert \leqslant 1.$ 
\end{enumerate}
\end{Proposition}

\subsection*{Remarks}
\begin{enumerate}
\item  There will be special roles in this paper for the exponents $q := \frac{2}{1 - 2 \varepsilon}$ and $q^\prime: = \frac{2}{1+ 2\varepsilon}$. Note that these are conjugate exponents ($\frac{1}{q} + \frac{1}{q^\prime} = 1$), $q>2$, and $\frac{1}{q^\prime}$ is the exponent appearing on the right-hand side of the inequality \eqref{eq:maintheoremconclusion}.
\item Speaking in graph-theoretic terms, the combinatorial structure in part (2) is a form of one-sided regularity for the bipartite graph whose vertex set is $V^{\prime} \sqcup W^{\prime}$ and whose edge set is $\mathcal{E}^{\prime}$. Indeed, suppose we denote the (weighted) density $\delta$ of this graph by \[ \delta: = \frac{\mu_{\psi,\theta}^{f,g}(\mathcal{E}^{\prime})}{\mu_{\psi}^f(V^{\prime}) \mu_{\theta}^g (W^{\prime})}.\] Given $v \in V^{\prime}$, the average (weighted) size of the neighbourhood $\Gamma_{\mathcal{E}^{\prime}}(v)$ is  $\delta \mu_{\theta}^g(W^{\prime})$. Part (2) says that for every such $v$ the size of the neighbourhood is at least $\frac{1}{q^{\prime}}$ times this expected value, and the analogous conclusion for every $w \in W^{\prime}$ too. We note that (2) is not a new manoeuvre in this problem, being closely modelled on \cite[Lemma 8.5]{KM19}. However, we have chosen to avoid graph-theoretic terminology in this article.
    \item 
The arithmetic structure from part (3) can be fruitfully compared to the example constructed by Duffin--Schaeffer \cite{DS41} (see the Appendix, or \cite{R17} for a more general approach), showing that Khintchine's Theorem is in general not true without the monotonicity condition.
Indeed, the proof of Proposition \ref{Prop:structureofminimalcounterexample} does not use the full strength of \eqref{eq:conv_bound}; it is enough that $\frac{f(n)}{n}$ and $\frac{g(n)}{n}$ are non-negative $1$-bounded multiplicative functions.
Since Proposition \ref{Prop:structureofminimalcounterexample} therefore can also be applied to 
$f(n) = g(n) = n$, the arithmetical structure part of Proposition \ref{Prop:structureofminimalcounterexample} can also be applied to the non-coprime
setting. We provide a more detailed discussion in the Appendix.

\end{enumerate}

Proposition \ref{Prop:structureofminimalcounterexample} is then combined with the following result. 
\begin{Proposition}[Resolution of minimal potential counterexample]
\label{Prop:resolvingminimalcounterexample}
Fix $\psi,\theta, \varepsilon, C, t, f,g$ satisfying the hypotheses of Theorem \ref{Theorem:maintheorem}. Suppose $\mathcal{E}^{\prime} \subset \mathcal{E}^{t,C}_{\psi,\theta}$ satisfies properties (2) and (3) from Proposition \ref{Prop:structureofminimalcounterexample} (with these parameters). Then $\mathcal{E}^{\prime}$ cannot satisfy property (1) of Proposition \ref{Prop:structureofminimalcounterexample}. 
\end{Proposition}

\noindent Manifestly the conclusion of Proposition \ref{Prop:resolvingminimalcounterexample} contradicts the conclusion of Proposition \ref{Prop:structureofminimalcounterexample}, so it follows that no counterexamples to Theorem \ref{Theorem:maintheorem} can exist. The proof of Proposition \ref{Prop:resolvingminimalcounterexample} is a direct computation, rather similar to equation \eqref{eq:DSexample} in the Appendix, but just using the condition $\vert \nu_p(v/N)\vert + \vert\nu_p(w/N)\vert \leqslant 1$ rather than any exact formula as in \eqref{eq:DSexample}. \\

It is worth reflecting briefly on how this strategy compares with the Koukoulopoulos--Maynard approach to proving \cite[Proposition 5.4]{KM19}. Koukoulopoulos--Maynard also used a two-stage strategy, beginning with a potential counterexample $\mathcal{E}$ and iterating until a highly structured counterexample $\mathcal{E}^{\prime}$ was found; this counterexample was then ruled out by direct calculation. The difference is that Koukoulopoulos--Maynard sought an immensely structured counterexample $\mathcal{E}^{\prime}$ (the `good GCD graph' of \cite[Proposition 7.1]{KM19}): it required fixed integers $a,b$ such that for all $(v,w) \in \mathcal{E}^\prime$, $a \vert v$, $b\vert w$ and $\gcd(a,b) = \gcd(v,w)$. This is a strictly stronger arithmetic structure than the one found in part (3) of Proposition \ref{Prop:structureofminimalcounterexample}. The presence of such integers $a$ and $b$ implies the structure found in part (3), taking $N = \gcd(a,b)$, but in our paper we do not require $N\vert v$ and $N \vert w$. 

This stronger structure makes proving the equivalent of Proposition \ref{Prop:resolvingminimalcounterexample} a little easier (this argument is found in \cite[Section 7]{KM19}). However, in order to locate such rigid structure, Koukoulopoulos--Maynard needed to fundamentally change the arithmetic weighting of the pairs $\mathcal{E}$, and iterate this altered quantity instead (introduced as the `quality' of a GCD graph in \cite[Defintion 6.6(d)]{KM19}). This alteration to the arithmetic weighting introduced a potentially unbounded loss, which then had to be won back elsewhere in the argument. This is the miraculous `cancellation of $\varphi$ factors' in expression \cite[(7.6)]{KM19}, in which a factor of $\frac{\varphi(a)\varphi(b)}{ab}$ cancels in numerator and denominator. 

Our approach via Propositions \ref{Prop:structureofminimalcounterexample} and \ref{Prop:resolvingminimalcounterexample} avoids this manoeuvre; in fact, the notion of GCD graphs and their `quality' is avoided completely. For readers familiar with the technical elements of \cite{KM19}, the particular part we avoid is the iteration on primes in $\mathcal{R}^{\sharp}$ that occurs in the delicate \cite[Proposition 8.2]{KM19}. The five cases (a)-(e) from \cite[pp 301]{KM19} do feature obliquely in our argument, via the five possible options for the pair of valuations $(\nu_p(v/N), \nu_p(w/N))$ allowed by part (3) of Proposition \ref{Prop:structureofminimalcounterexample}, namely $\{(-1,0), (0,-1), (0,0), (0,1), (1,0)\}$. Yet these cases can be handled in a unified way during the proof of Proposition \ref{Prop:resolvingminimalcounterexample}, keeping the proof relatively efficient. The structure developed by Koukoulopoulos--Maynard using the other quality iterations can be efficiently packaged in Lemma \ref{Lemma:bilinearbound} and Lemma \ref{Lemma:decayawayfromdiagonal} below, which recast the argument as a structure theorem for measures on $\mathbb{Z} \times \mathbb{Z}$ (with no iteration required). 

To mention another complication present in \cite{KM19} that we avoid, Koukoulopoulos--Maynard had to carefully control the number of iterations that occurred in their method, for fear of weakening the condition of $p \vert vw/\gcd(v,w)^2$ too much through the iteration. This is the source of the separation of conditions (d)(i) and (d)(ii) in \cite[Proposition 7.1]{KM19}. We absorb this concern into the `minimal counterexample' property in Proposition \ref{Prop:structureofminimalcounterexample}, thus avoiding a case analysis and leading to a further simplification of the argument.

\section{Proof of Proposition \ref{Prop:structureofminimalcounterexample}}
Take $\psi,\theta,\varepsilon,C,t,f,g,\mathcal{E},V,W$ as in the statement of Proposition \ref{Prop:structureofminimalcounterexample}, i.e.\! forming a minimal counterexample to Theorem \ref{Theorem:maintheorem}. 
Suppose $p \in \mathcal{P}_{\psi,\theta}$ \ed{(note that trivially, the conclusion of Theorem \ref{Theorem:maintheorem} is true for 
$ \mathcal{P}_{\psi,\theta} = \emptyset$, thus we may assume $|\mathcal{P}_{\psi,\theta}| \geq 1$)}. For $i,j \geqslant 0$, define $V_{i}: = \{v \in V: \, p^i \Vert v\}$ and $W_{j}: = \{w \in W: \, p^j \Vert w\}$. Since $\mathcal{E}$ is a counterexample to Theorem \ref{Theorem:maintheorem} we must have $\mu_{\psi,\theta}^{f,g}(\mathcal{E}) >0$, so we may define \[ m(i,j): = \frac{\mu_{\psi,\theta}^{f,g}(\mathcal{E} \cap (V_i \times W_j))}{\mu_{\psi,\theta}^{f,g}(\mathcal{E})}.\] We have $m(i,j) \geqslant 0$, and \[ \sum_{i,j \geqslant 0} m(i,j) = 1,\] meaning that $m$ defines a finitely-supported measure on $\mathbb{Z}_{ \geqslant 0} \times \mathbb{Z}_{\geqslant 0}$. 

Using the minimality assumption on $\mathcal{E}$, we may upper-bound $m(i,j)$ by a bilinear expression with exponential decay away from the diagonal $i=j$. 

\begin{Lemma}[Bilinear upper bound]
\label{Lemma:bilinearbound}
Let $\psi,\theta,\varepsilon,C,t,f,g,\mathcal{E},m(i,j)$ be as above.
Writing for $i,j \geqslant 0$, $\alpha_i: = \mu_{\psi}^f(V_i)/ \mu_{\psi}^f(V)$ and $\beta_j: = \mu_{\theta}^g(W_j)/\mu_{\theta}^g(W)$, we have
\[ m(i,j) \leqslant 1000^{-\mathds{1}_{[p \leqslant p_0(\varepsilon)]}}  p^{-\frac{\vert i-j\vert}{q}}   (\alpha_i \beta_j e^{ \mathds{1}_{[i \neq j]}}) ^{\frac{1}{q'}}.\]
\end{Lemma}
\begin{proof}
\ed{Note that if $f(p^i) = 0$ or $g(p^j) = 0$, then $m(i,j) = 0$, thus the statement holds trivially. Hence we may assume from now on that $f(p^i) > 0$ and $g(p^j) > 0$.}
We will bound $m(i,j)$ by bounding $\mu_{\psi,\theta}^{f,g}(\mathcal{E})$ from below and $\mu_{\psi,\theta}^{f,g}(\mathcal{E} \cap (V_i \times W_j))$ from above. The na\"{i}ve combination of these bounds will control $m(i,j)$ as required. 

Indeed, since $\mathcal{E}$ is a counterexample to Theorem \ref{Theorem:maintheorem} we have \begin{equation}
\label{eq:lowerboundonE}
\mu_{\psi,\theta}^{f,g}(\mathcal{E}) \geqslant 1000^{P_{\psi,\theta}(\varepsilon)}\left( \mu_{\psi}^f(V) \mu_{\theta}^g(W) e^{-Ct}\right)^{\frac{1}{2} + \varepsilon}.
\end{equation} In order to bound $\mu_{\psi,\theta}^{f,g}(\mathcal{E} \cap (V_i \times W_j))$ from above we will remove the influence of the prime $p$ and then use the minimality assumption. The main barrier is notational.  Define $\widetilde{\psi_{i,j}}, \widetilde{\theta_{i,j}}: \mathbb{N} \longrightarrow \mathbb{R}_{\geqslant 0}$ by \[\widetilde{\psi_{i,j}}(v) := \begin{cases}
p^{j - \min(i,j)}\psi(p^i v) & \text{if } \gcd(p,v) = 1, \\
0 & \text{if } p \vert v, \end{cases}\] and\[\widetilde{\theta_{i,j}}(w) := \begin{cases}
p^{i - \min(i,j)}\theta(p^jw) & \text{if } \gcd(p,w) = 1, \\
0 & \text{if } p \vert w, \end{cases}\] with
 \begin{align*}
\widetilde{\mathcal{E}_{i,j}} := \{(v,w): (p^i v, p^j w) \in \mathcal{E} \cap (V_i \times W_j)\},\quad \widetilde{V_i}:= \{ v: \, p^iv \in V_i\},\quad
\widetilde{W_j}:= \{w: \, p^j w \in W_j\}.
\end{align*} Then 
\begin{align*}
\mu_{\widetilde{\psi_{i,j}}}^f(\widetilde{V_i}) = p^{j - \min(i,j)}\sum\limits_{v \in \widetilde{V_i}} \frac{ f(v) \psi(p^i v)}{v}= p^{j- \min(i,j)} \Big(\frac{p^i}{f(p^i)}\Big)\mu_{\psi}^f(V_i),
\end{align*}
and similarly
\begin{align*}
\mu_{\widetilde{\theta_{i,j}}}^g(\widetilde{W_j}) = p^{i - \min(i,j)} \Big( \frac{p^j}{g(p^j)}\Big)\mu_{\theta}^g(W_j).
\end{align*}
Also
\begin{align*}
\mu_{\psi,\theta}^{f,g}(\mathcal{E} \cap (V_i \times W_j))
&=\Big( \frac{ f(p^i) g(p^j)}{p^{i+j}}\Big)p^{-\vert i-j\vert} \mu_{\widetilde{\psi_{i,j}}, \widetilde{\theta_{i,j}}}^{f,g}(\widetilde{\mathcal{E}_{i,j}}).
\end{align*} 

The set $\widetilde{\mathcal{E}_{i,j}}$ is of the type that is considered in Theorem \ref{Theorem:maintheorem}, enabling us to use the minimality assumption on $\mathcal{E}$ to upper-bound the contribution from $\widetilde{\mathcal{E}_{i,j}}$. Indeed, if $(v,w) \in \widetilde{\mathcal{E}_{i,j}}$ then neither $v$ nor $w$ have a factor of $p$ and
\begin{align*}
\gcd(p^i v, p^j w) \geqslant \max(p^j w\, \psi(p^i v), p^i v\, \theta(p^j w)).
\end{align*}
Therefore 
\begin{align*}
\gcd(v,w) &\geqslant \max( p^{j - \min(i,j)}w\, \psi(p^i v), p^{i - \min(i,j)} v\, \theta (p^jw))\\
&= \max( w\, \widetilde{\psi_{i,j}}(v), v \,\widetilde{\theta_{i,j}}(w)).
\end{align*}
In addition,
\begin{align*}
\sum\limits_{\substack{r \geqslant t \\ r \vert \frac{p^{i+j}vw}{\gcd(p^iv,p^jw)^2}}} \frac{1}{r} \geqslant C,
\end{align*}
where $r$ runs over primes here. Hence 
\begin{align*}
\sum\limits_{\substack{r \geqslant t \\ r \vert \frac{vw}{\gcd(v,w)^2}}} \frac{1}{r} \geqslant C - \frac{\mathds{1}_{[i \neq j]}}{p} \geqslant C - \frac{\mathds{1}_{[i \neq j]}}{t}. 
\end{align*}
So $\widetilde{\mathcal{E}_{i,j}} \subset \mathcal{E}^{t, C - \frac{\mathds{1}_{[i \neq j]}}{t}}_{\widetilde{\psi_{i,j}}, \widetilde{\theta_{i,j}}}$. Finally \[ \mathcal{P}_{\widetilde{\psi_{i,j}}, \widetilde{\theta_{i,j}}} \ed{\;\subset\;} \mathcal{P}_{\psi,\theta} \setminus \{p\}.\] Therefore, by the minimality assumption on $\vert \mathcal{P}_{\psi,\theta}\vert$, the bound \eqref{eq:maintheoremconclusion} holds for $\mu_{\widetilde{\psi_{i,j}}, \widetilde{\theta_{i,j}}}^{f,g}(\widetilde{\mathcal{E}_{i,j}})$. Hence,
\begin{align*}
&\mu_{\psi,\theta}^{f,g}(\mathcal{E} \cap (V_i \times W_j)) \\
&=\Big(\frac{f(p^i) g(p^j)}{p^{i+j}}\Big)p^{-\vert i-j\vert} \mu_{\widetilde{\psi_{i,j}}, \widetilde{\theta_{i,j}}}^{f,g}(\widetilde{\mathcal{E}_{i,j}}) \\
&\leqslant \Big(\frac{f(p^i) g(p^j)}{p^{i+j}}\Big)p^{-\vert i-j\vert} 1000^{P_{\psi,\theta}(\varepsilon) - \mathds{1}_{[p \leqslant p_0(\varepsilon)]}}(\mu_{\widetilde{\psi_{i,j}}}^f(\widetilde{V_i})\mu_{\widetilde{\theta_{i,j}}}^g(\widetilde{W_j}) e^{-t(C- \frac{\mathds{1}_{[i \neq j]}}{t})})^{\frac{1}{2} + \varepsilon}\\
&=\Big(\frac{f(p^i) g(p^j)}{p^{i+j}}\Big)^{\frac{1}{2} - \varepsilon} p^{-\vert i-j\vert (\frac{1}{2} - \varepsilon)}1000^{P_{\psi,\theta}(\varepsilon) - \mathds{1}_{[p \leqslant p_0(\varepsilon)]}}( \mu_{\psi}^f( V_i) \mu_{\theta}^g(W_j) e^{-t(C- \frac{\mathds{1}_{[i \neq j]}}{t})})^{\frac{1}{2} + \varepsilon} \\
&\leqslant  p^{-\vert i-j\vert (\frac{1}{2} - \varepsilon)}1000^{P_{\psi,\theta}(\varepsilon) - \mathds{1}_{[p \leqslant p_0(\varepsilon)]}}( \mu_{\psi}^f( V_i) \mu_{\theta}^g(W_j) e^{-Ct +\mathds{1}_{[i \neq j]}})^{\frac{1}{2} + \varepsilon},
\end{align*}
since $f(p^i) \leqslant p^i$ and $g(p^j) \leqslant p^j$. 

Combining with the lower bound on $\mu_{\psi,\theta}^{f,g}(\mathcal{E})$ from equation \eqref{eq:lowerboundonE}, we obtain \[ m(i,j) \leqslant 1000^{-\mathds{1}_{[p \leqslant p_0(\varepsilon)]}}  p^{-\vert i-j\vert (\frac{1}{2} - \varepsilon)}  (\alpha_i \beta_j e^{\mathds{1}_{[i \neq j]}}) ^{\frac{1}{2} + \varepsilon},\] as required. 
\end{proof}
To make use of this bilinear upper bound, we will adapt a structural lemma on such measures from work of Green and the third named author \cite{GW21}. 
\begin{Lemma}[Decay away from the diagonal]
\label{Lemma:decayawayfromdiagonal}
Let $q >2$, and write $q^\prime$ for the conjugate index to $q$ (i.e.\! $\frac{1}{q} + \frac{1}{q^\prime} = 1$). Let $c_1 \leqslant 1$, $0 < c_2 < 1$, $\lambda  \in (0, 1-c_2]$, and $C_3 >0$. Suppose that $m$ is a finitely-supported probability measure on $\mathbb{Z}^2$. Suppose that there are sequences $x = (x_i)_{i \in \mathbb{Z}}$, $y = (y_j)_{j \in \mathbb{Z}}$ of non-negative reals such that $\Vert x\Vert_{\ell^{q^\prime}(\mathbb{Z})} = \Vert y \Vert_{\ell^{q^\prime}(\mathbb{Z})} = 1$, and such that for all $(i,j) \in \mathbb{Z}^2$ we have 
\begin{equation}
\label{eq:generalbilinearbound}
 m(i,j) \leqslant \begin{cases}
c_1 x_i y_j & \text{if } i =j, \\
c_1 C_3 \lambda^{\vert i-j\vert} x_i y_j & \text{if } i \neq j.\end{cases}
\end{equation} Then $c_1 \geqslant \frac{c_2}{1 + (2 C_3 - 1) \lambda}$, and there exists $k \in \mathbb{Z}$ such that 
\begin{equation}
\label{eq:concentrationinequality} 
\sum\limits_{\vert i-k\vert + \vert j-k\vert \geqslant 2} m(i,j) \ll_{q, c_2, C_3} (\lambda^q)^{\frac{2}{q^\prime}} + (\lambda^{q})^{1 + \frac{1}{q}}.
\end{equation}
\end{Lemma}
\subsection*{Remarks}
\begin{enumerate}
\item We may take $k$ to be an index for which $x_ky_k = \sup_i x_i y_i$. The position $m(k,k)$ is a natural choice for the place where $m(i,j)$ is concentrated, as $(k,k)$ is the index for which the bilinear upper bound \eqref{eq:generalbilinearbound} is weakest.
\item The exponents in \eqref{eq:concentrationinequality} are written to indicate how each exponent is strictly greater than $q$. This is the critical property for the application. 
\item The implied constant in \eqref{eq:concentrationinequality} may be calculated, and tends to infinity as either $q \to 2$, $c_2 \to 0$, or $C_3 \to \infty$.
\item The proof of Lemma \ref{Lemma:decayawayfromdiagonal} breaks down completely if $c_1 > 1$. It was therefore vital in Lemma \ref{Lemma:bilinearbound} that the multiplicative loss of a factor of $e$ in the upper bound occurred only on the off-diagonal $i \neq j$ terms. This corresponds to the anatomy condition on primes $p \vert vw/\gcd(v,w)^2$ only involving the primes that divide $v$ and $w$ to different exponents. 
\end{enumerate}
\noindent

\begin{proof}[Proof of Lemma \ref{Lemma:decayawayfromdiagonal}]
This is very similar to \cite[Lemma 2.1]{GW21}. The difference is the introduction of the parameter $C_3$. 

We first prove the lower bound on $c_1$. Indeed, since $\sum_{i,j} m(i,j) = 1$ and $q^\prime \in (1,2)$ we have 
\begin{align*}
 \frac{1}{c_1} \leqslant \sum_{i \in \mathbb{Z}} x_i y_i + C_3\sum_{i \neq j} \lambda^{ \vert i-j\vert} x_i y_j &\leqslant \sum_{i \in \mathbb{Z}} (x_i y_i)^{\frac{q^\prime}{2}} + C_3\sum_{\substack{m \in \mathbb{Z} \\ m \neq 0}} \lambda^{\vert m\vert}\sum_{i \in \mathbb{Z}} (x_i y_{i+m})^{\frac{q^\prime}{2}} \\
&\leqslant \Big( 1 + C_3 \sum_{\substack{m \in \mathbb{Z} \\ m \neq 0}} \lambda^{\vert m\vert}\Big)\cdot \sup_{m \in \mathbb{Z}} \sum_{i \in \mathbb{Z}} (x_i y_{i+m})^{\frac{q^\prime}{2}} \leqslant \frac{1 + (2C_3 - 1)\lambda}{1 - \lambda} \cdot 1,
\end{align*} where the final inequality followed by Cauchy--Schwarz and the assumption that $\Vert x\Vert_{\ell^{q^\prime}(\mathbb{Z})} = \Vert y \Vert_{\ell^{q^\prime}(\mathbb{Z})} = 1$. The lower bound on $c_1$ follows from rearranging.

Turning to \eqref{eq:concentrationinequality}, write $\sup_i x_iy_i = 1 - \gamma$ for some $\gamma \in [0,1]$, and suppose this supremum is attained when $i=k$ (such an index $k$ exists since
$\sum_{i \in \mathbb{Z}}x_iy_i$ converges). Then $x_k,y_k \geqslant 1- \gamma$, and so $x_k^{q^\prime}, y_k^{q^\prime} \geqslant 1 - q^\prime\gamma$. Therefore
\begin{equation}
\label{equation:easynonsupbound}
 \sum_{i \neq k} x_i^{q^\prime}, \sum_{j \neq k} y_j^{q^\prime} \leqslant q^\prime\gamma \leqslant 2\gamma \text{ and } x_i,y_j \leqslant (q^\prime)^{\frac{1}{q^\prime}}\gamma^{\frac{1}{q^\prime}} \leqslant 2\gamma^{\frac{1}{q^\prime}} \text{ when } i,j \neq k.
\end{equation}
For $n = 1,2,3,4,5,6$ write $\sum_n: = \sum_{(i,j) \in S_n} m(i,j)$, where $S_1,\dots,S_6$ are the following sets, which partition $\mathbb{Z}^2$:
\[ S_1 : = \{(i,j) \in \mathbb{Z}^2: i \neq j\neq k \neq i\}, \qquad S_2: = \{(k,j): \, \vert j-k\vert \geqslant 2\},\] \[ S_3: = \{(i,k): \, \vert i-k\vert \geqslant 2\}, \qquad S_4: = \{(k,k \pm 1), (k \pm 1, k)\}, \qquad S_5: = \{(i,i): i \neq k\},\] and finally $S_6: = \{(k,k)\}$. We bound the $\sum_n$ in turn. \\

\textbf{Bound for $\sum_1$}. By \eqref{equation:easynonsupbound}, if $i \neq k$ we have $x_i \ll x_i^{\frac{q^\prime}{2}} \gamma^{\frac{1}{q^\prime} (1 - \frac{q^\prime}{2})}$ (and analogously for $y_j$ if $j \neq k$). Hence \[ \textstyle{\sum_1} \leqslant C_3 \sum\limits_{\substack{i,j \neq k \\ i \neq j}} \lambda^{\vert i-j\vert} x_i y_j  \ll_{C_3}  \gamma^{\frac{2}{q^\prime} (1 - \frac{q^\prime}{2})} \sum\limits_{\substack{i,j \neq k \\ i \neq j}} \lambda^{\vert i-j\vert} (x_i y_j)^{\frac{q^\prime}{2}}= \gamma^{\frac{2}{q^\prime} - 1} \sum_{m \neq 0} \lambda^{\vert m\vert} \sum\limits_{i, i+m \neq k} (x_i y_{i+m})^{\frac{q^\prime}{2}}.\] By Cauchy--Schwarz and \eqref{equation:easynonsupbound}, for each fixed $m$ we have \[ \sum\limits_{i,i+m \neq k}(x_i y_{i+m})^{\frac{q^\prime}{2}} \leqslant \Big( \sum\limits_{i \neq k} x_i^{q^\prime}\Big)^{1/2} \Big( \sum\limits_{j \neq k} y_j^{q^\prime}\Big)^{1/2} \ll \gamma.\] \ed{Furthermore $\sum_{m \neq 0} \lambda^{\vert m\vert} = \frac{2\lambda}{1 - \lambda}  \leqslant \frac{2\lambda}{c_2}$. Therefore $\sum_1 \ll_{c_2,C_3}  \lambda\gamma^{\frac{2}{q^\prime}}$.} \\

\textbf{Bounds for $\sum_2$, $\sum_3$}. For $\sum_2$, we use the trivial bound $x_k \leqslant 1$ to derive \[ \textstyle{\sum_2} \leqslant C_3\sum\limits_{\vert j-k\vert \geqslant 2} \Big(\frac{\lambda}{1 - c_2}\Big)^{\vert j-k\vert} (1 - c_2)^{\vert j-k\vert} y_j \leqslant C_3 \Big( \frac{\lambda}{1 - c_2}\Big)^{2} \sum\limits_{\vert j-k\vert \geqslant 2} (1 - c_2)^{\vert j-k\vert} y_j.\] Continuing, using \eqref{equation:easynonsupbound} to bound $y_j$, we get 
\begin{align*}
 \sum\limits_{\vert j-k\vert \geqslant 2} (1 - c_2)^{\vert j-k\vert} y_j \ll \gamma^{\frac{1}{q^\prime}(1 - \frac{q^\prime}{2})} \sum\limits_{\vert j-k\vert \geqslant 2} (1 - c_2)^{\vert j-k\vert} y_j^{\frac{q^\prime}{2}} &\ll \gamma^{\frac{1}{q^\prime}(1 - \frac{q^\prime}{2})} \Big( \sum\limits_{j \neq k} (1 - c_2)^{2 \vert j-k\vert}\Big)^{\frac{1}{2}} \Big( \sum_{j \neq k} y_j^{q^\prime}\Big)^{\frac{1}{2}}\\
&\ll\gamma^{\frac{1}{q^\prime}} \Big( \sum\limits_{j \neq k} (1 - c_2)^{2 \vert j-k\vert}\Big)^{\frac{1}{2}}.
\end{align*}
Computing this geometric series, we obtain the overall bound \[ \textstyle{\sum_2} \ll_{c_2,C_3} \gamma^{\frac{1}{q^\prime}}\lambda^2 .\] An identical argument yields the same bound for $\sum_3$. \\

\textbf{Bound for $\sum_4$}. From \eqref{equation:easynonsupbound} we immediately get $\sum_4 \ll_{C_3} \gamma^{\frac{1}{q^\prime}} \lambda$. \\

\textbf{Bound for $\sum_5$}. Arguing similarly to the $\sum_1$ case, we have \[ \textstyle{\sum_5} \leqslant \sum_{i \neq k} x_iy_i \ll \gamma^{\frac{2}{q^\prime} (1 - \frac{q^\prime}{2})} \sum_{i \neq k} (x_i y_i)^{\frac{q^\prime}{2}} \ll \gamma^{\frac{2}{q^\prime} - 1}\Big( \sum\limits_{i \neq k} x_i^{q^\prime}\Big)^{1/2} \Big( \sum\limits_{j \neq k} y_j^{q^\prime}\Big)^{1/2} \ll \gamma^{\frac{2}{q^\prime}}.\] 

\textbf{Bound for $\sum_5 + \sum_6$}. By the fact that $\sup_i x_i y_i = 1 - \gamma$, we have \[ \textstyle{ \sum_5 + \sum_6}  \leqslant \sum_i x_i y_i \leqslant (1 - \gamma)^{1 - \frac{q^\prime}{2}} \sum_i (x_i y_i)^{\frac{q^\prime}{2}} \leqslant (1 - \gamma)^{1 - \frac{q^\prime}{2}} \leqslant 1 - (1 - \frac{q^\prime}{2}) \gamma,\] where again we used Cauchy--Schwarz. \\

Putting everything together we have 
\begin{align*}
 1 = \sum_{n=1}^6 \textstyle{\sum_n}  &\leqslant 1 - (1 - \tfrac{q^\prime}{2}) \gamma + O_{c_2,C_3}(\gamma^{\frac{1}{q^\prime}} \lambda). 
\end{align*}
Since $q > 2$, we have $q^\prime <2$. Rearranging the above, and using $1 - \frac{1}{q^\prime} = \frac{1}{q}$, we get 
\begin{equation}
\label{eq:gammaboundbylambda}
 \gamma \ll_{c_2,C_3,q} \lambda^q. 
\end{equation}
\noindent Finally, \begin{align*}
\sum\limits_{\vert i-k\vert + \vert j-k\vert \geqslant 2} m(i,j) &= \textstyle{\sum_1 + \sum_2 + \sum_3 + \sum_5} \ll_{c_2,C_3} \gamma^{\frac{2}{q^\prime}} + \gamma^{\frac{1}{q^\prime}} \lambda^2. 
\end{align*}
Substituting the bound \eqref{eq:gammaboundbylambda} for the $\gamma$ terms in the above, we obtain the claimed bound.
\end{proof}
We now apply Lemma \ref{Lemma:decayawayfromdiagonal} to the bound given in Lemma \ref{Lemma:bilinearbound}. In this case, $\lambda = p^{-\frac{1}{2} + \varepsilon}$, $x_i = (\alpha_i)^{\frac{1}{2} + \varepsilon}$, $y_j = (\beta_j)^{\frac{1}{2} + \varepsilon}$, $q = \frac{2}{1 - 2 \varepsilon}$, $q^\prime = \frac{2}{1 + 2 \varepsilon}$, $c_1 = 1000^{-\mathds{1}_{[p \leqslant p_0(\varepsilon)]}}$, and $C_3 = e$. Finally, since $\varepsilon \leqslant 2/5$ we may take $c_2 = 1 - 2^{-\frac{1}{2} + \frac{2}{5}} = 1 - 2^{-\frac{1}{10}}$. Extending $x_i, y_j, m(i,j)$ to negative integers by extending by $0$, and since $\sum_i \alpha_i = \sum_j \beta_j = 1$, these sequences satisfy the hypotheses $\Vert x\Vert_{\ell^{q^\prime}(\mathbb{Z})} = \Vert y \Vert_{\ell^{q^\prime}(\mathbb{Z})} = 1$.

The first conclusion of Lemma \ref{Lemma:decayawayfromdiagonal} then implies that \[1000^{-\mathds{1}_{[p \leqslant p_0(\varepsilon)]}} = c_1 \geqslant \frac{c_2}{1 + (2C_3 - 1) \lambda} \geqslant \frac{c_2}{2 C_3} \geqslant \frac{1 - 2^{-1/10}}{ 2e } > 1/1000.\] Therefore $p > p_0(\varepsilon)$. Next, we have the existence of some $k_p \in \mathbb{Z}_{ \geqslant 0}$ such that \[\sum\limits_{\vert i-k_p\vert + \vert j-k_p\vert \geqslant 2} m(i,j) \ll_{\varepsilon} \lambda^{\frac{2q}{q^\prime}} + \lambda^{2 + \frac{q}{q^\prime}} = p^{-1 - 2\varepsilon} + p^{-\frac{3}{2} + \varepsilon} \leqslant p^{-1 - 2 \varepsilon} + p^{-\frac{11}{10}}\] as $\varepsilon \leqslant 2/5$. It is at this point in the argument where we see the importance of having $\varepsilon$ bounded away from $1/2$, so that both these exponents are strictly less than $-1$. 

By considering all primes $p \in \mathcal{P}_{\psi,\theta}$ (which we have shown must all be at least $p_0(\varepsilon)$), and taking the union bound, we conclude that 
\begin{align*}
&\mu_{\psi,\theta}^{f,g}( \{(v,w) \in \mathcal{E}: \, \exists p \in \mathcal{P}_{\psi,\theta} \, \text{s.t.} \, \vert \nu_p(v) - k_p \vert + \vert \nu_p(w) - k_p\vert \geqslant 2\})\nonumber\\
& \ll_{\varepsilon} \mu_{\psi, \theta}^{f,g}(\mathcal{E}) \sum_{p > p_0(\varepsilon)} (p^{-1 - 2 \varepsilon} + p^{-\frac{11}{10}}).\nonumber
\end{align*}
Since $p_0(\varepsilon)$ is sufficiently large, we conclude that 
\begin{equation}
\label{eq:removingbadpairs}
\mu_{\psi,\theta}^{f,g}( \{(v,w) \in \mathcal{E}: \, \exists p \in \mathcal{P}_{\psi,\theta} \, \text{s.t.} \, \vert \nu_p(v) - k_p \vert + \vert \nu_p(w) - k_p\vert \geqslant 2\}) <\frac{\mu_{\psi, \theta}^{f,g}(\mathcal{E})}{2}. 
\end{equation}

We now construct $\mathcal{E}^{\prime}$. Let $N = \prod_{p \in \mathcal{P}_{\psi,\theta}} p^{k_p}$, and begin by considering \[ \mathcal{E}^{*}: = \{(v,w) \in \mathcal{E}: \, \text{for all primes } p, \, \vert \nu_p(v/N)\vert + \vert \nu_p(w/N)\vert \leqslant 1\}.\] Let $V^*: = \mathcal{E}^*|_V$ and $W^*: = \mathcal{E}^*|_W$. Since $\vert \nu_p(v/N)\vert + \vert \nu_p(w/N)\vert = 0$ for all primes $p \notin \mathcal{P}_{\psi,\theta}$ and all $(v,w) \in V \times W$, the bound \eqref{eq:removingbadpairs} implies 
\begin{align*}
\mu_{\psi,\theta}^{f,g}(\mathcal{E}^{*}) \geqslant \frac{\mu_{\psi, \theta}^{f,g}(\mathcal{E})}{2} &> \frac{1}{2} \cdot 1000^{P_{\psi,\theta}(\varepsilon)}\left( \mu_{\psi}^f(V) \mu_{\theta}^g(W) e^{-Ct}\right)^{\frac{1}{q^\prime}} \\
& \geqslant \frac{1}{2} \cdot 1000^{P_{\psi,\theta}(\varepsilon)}\left( \mu_{\psi}^f(V^{*}) \mu_{\theta}^g(W^{*}) e^{-Ct}\right)^{\frac{1}{q^\prime}}.
\end{align*}
Now define $\mathcal{E}^{\prime} \subset \mathcal{E}^{*}$ to be minimal (with respect to inclusion) 
such that property (1) of Proposition \ref{Prop:structureofminimalcounterexample} is satisfied. \ed{In other words, $\mathcal{E}'$ is an arbitrary subgraph of $\mathcal{E}^*$ that satisfies property (1) of Proposition \ref{Prop:structureofminimalcounterexample}, and any proper subgraph
$\mathcal{E}'' \subsetneq \mathcal{E}'$ does not satisfy property (1) of Proposition \ref{Prop:structureofminimalcounterexample}. Note that such a graph might not be uniquely defined, but we might choose an arbitrary one among all subgraphs $\mathcal{E}'$ with that property. Note that by construction, $\mathcal{E}^{\prime}$ exists, since $\mathcal{E}^*$ is finite and satisfies property (1) itself.} $\mathcal{E}^{\prime}$ also satisfies property (3) (since $\mathcal{E}^*$ satisfies property (3)). It remains to demonstrate property (2). 

Indeed, letting $V^{\prime}: = \mathcal{E}^{\prime} |_V$ and  $W^{\prime}: = \mathcal{E}^{\prime} |_W$, suppose for contradiction that there exists some $v \in V^{\prime}$ for which 
\begin{equation}
\label{eq:upperboundforcontradiction}
\mu_{\theta}^g(\Gamma_{\mathcal{E}^{\prime}}(v)) < \frac{1}{q^\prime} \frac{\mu_{\psi,\theta}^{f,g}(\mathcal{E}^{\prime})}{\mu_{\psi}^f(V^{\prime})}.
\end{equation}
(The case with the contradiction coming from some $w \in W^{\prime}$ is identical). Consider \[ \mathcal{E}^{\text{new}} := \mathcal{E}^{\prime} \cap \left((V^{\prime} \setminus \{v\}) \times W^{\prime}\right),\] and define $V^{\text{new}}: = \mathcal{E}^{\text{new}} |_V$ and $W^{\text{new}}: = \mathcal{E}^{\text{new}} |_W$. Since $v \in V^{\prime} = \mathcal{E}^\prime|_V$, there must be some $w \in W^{\prime}$ such that $(v,w) \in \mathcal{E}^\prime$
and thus $\vert \mathcal{E}^{\text{new}} \vert < \vert \mathcal{E}^{\prime}\vert$. Yet $\mathcal{E}^{\text{new}}$ also satisfies property (1) of Proposition \ref{Prop:structureofminimalcounterexample}, contradicting the minimality of $\mathcal{E}^{\prime}$. Indeed, observe first that $\mu_{\psi}^f(V^{\text{new}}) \neq 0$, as otherwise \[\mu_{\psi}^f(v)\mu_{\theta}^g(\Gamma_{\mathcal{E}^{\prime}}(v)) = \mu_{\psi, \theta}^{f,g}(\mathcal{E}^{\prime}),\] contradicting the upper bound \eqref{eq:upperboundforcontradiction}. Then 
\begin{align*}
 \mu_{\psi,\theta}^{f,g}(\mathcal{E}^{\text{new}}) &= \mu_{\psi,\theta}^{f,g}(\mathcal{E}^{\prime}) - \mu_{\psi}^f(v) \mu_{\theta}^g(\Gamma_{\mathcal{E}^{\prime}}(v)) \\
&>  \mu_{\psi,\theta}^{f,g}(\mathcal{E}^{\prime}) - \frac{\mu_{\psi}^f(v) }{q^\prime} \frac{ \mu_{\psi,\theta}^{f,g}(\mathcal{E}^{\prime})}{\mu_{\psi}^f(V^{\prime})}\\
&=\frac{\mu_{\psi,\theta}^{f,g}(\mathcal{E}^{\prime})}{\mu_{\psi}^f(V^{\prime})}\mu_{\psi}^f(V^{\text{new}})\Big(1 + \frac{ \mu_{\psi}^f(v)}{q\, \mu_{\psi}^{f}(V^{\text{new}})}\Big)\\
&\geqslant \frac{\mu_{\psi,\theta}^{f,g}(\mathcal{E}^{\prime})}{\mu_{\psi}^f(V^{\prime})}\mu_{\psi}^f(V^{\text{new}})\Big(1 + \frac{ \mu_{\psi}^f(v)}{ \mu_{\psi}^f(V^{\text{new}})}\Big)^{\frac{1}{q}}\\
&= \mu_{\psi,\theta}^{f,g}(\mathcal{E}^{\prime}) \Big( \frac{ \mu_{\psi}^f(V^{\text{new}})}{\mu_{\psi}^f(V^{\prime})}\Big)^{\frac{1}{q^\prime}}.
\end{align*}
Hence \[ \mu_{\psi,\theta}^{f,g}(\mathcal{E}^{\text{new}}) \left(\mu_{\psi}^f(V^{\text{new}})\mu_{\theta}^g(W^{\text{new}})\right)^{-\frac{1}{q^\prime}} > \mu_{\psi,\theta}^{f,g}(\mathcal{E}^{\prime}) \left(\mu_{\psi}^f(V^{\prime})\mu_{\theta}^g(W^{\text{new}})\right)^{-\frac{1}{q^\prime}} \geqslant \left(\mu_{\psi,\theta}^{f,g}(\mathcal{E}^{\prime}) (\mu_{\psi}^f(V^{\prime})\mu_{\theta}^g(W^{\prime})\right)^{-\frac{1}{q^\prime}}.\] Since $\mathcal{E}^{\prime}$ satisfies property (1) of Proposition \ref{Prop:structureofminimalcounterexample} by construction, we conclude that $\mathcal{E}^{\text{new}}$ also satisfies property (1). But this contradicts the minimality of $\mathcal{E}^{\prime}$. 

Therefore $\mathcal{E}^{\prime} \subset \mathcal{E}$ satisfies properties (1), (2) and (3), and Proposition \ref{Prop:structureofminimalcounterexample} is resolved. \qed

\section{Anatomy lemmas}
\label{Section:anatomy}
In this section, we present some statements concerning the anatomic conditions occurring in the statement of Theorem \ref{Theorem:maintheorem}. We remark that while Lemma \ref{Lemma:unweightedanatomy} essentially appears in \cite[Lemma 7.3]{KM19}, Lemma \ref{Lemma:divisoranatomy} is a new ingredient in our approach.

\begin{Lemma}[Unweighted anatomy property]
\label{Lemma:unweightedanatomy}
For any real $x,t \geqslant 1$ and $c \in \mathbb{R}$, 
\[\Big\vert \{n \leqslant x: \, \sum\limits_{\substack{p \geqslant t \\ p \vert n}} \frac{1}{p} \geqslant c \}\Big\vert \ll x e^{-100ct}.\]
\end{Lemma}
\begin{proof}
\ed{This is \cite[Lemma 10]{ABH23} for $c >0$. However if $c\leqslant 0$, the statement is trivial.}
\end{proof}
\begin{Lemma}[Divisor anatomy property]
\label{Lemma:divisoranatomy}
Let $f$ be a non-negative multiplicative function that satisfies \eqref{eq:conv_bound}. Then for any $M \in \mathbb{N}$, real $t \geqslant 1$ and $c \in \mathbb{R}$, \[\sum\limits_{\substack{mn = M \\ \sum\limits_{\substack{p \geqslant t \\ p\vert m}} \frac{1}{p} \geqslant c}} f(n) \ll Me^{-100ct}.\]
\end{Lemma}
\begin{proof}
If $c<0$ the statement is trivial as $\sum_{n\vert M} f(n) \leqslant M$, so assume that $c\geqslant 0$. Then 
\begin{align*}
\sum\limits_{\substack{mn = M \\ \sum\limits_{\substack{p \geqslant t \\ p\vert m}} \frac{1}{p} \geqslant c}} f(n) &\leqslant \exp(-100ct) \sum\limits_{\substack{mn = M}} \exp\Big(\sum\limits_{\substack{p \geqslant t \\ p\vert m}} \frac{100 t}{p}\Big)f(n)\\
&= e^{-100ct} \prod\limits_{\substack{p < t \\ p\vert M}} \Big((1 \star f)(p^{\nu_p(M)}) \Big) \prod_{\substack{p \geqslant t \\ p\vert M}}\Big(f(p^{\nu_p(M)}) + e^{\frac{100 t}{p}} (1 \star f)(p^{\nu_p(M) - 1})\Big)\\
&=Me^{-100ct} \prod\limits_{\substack{p < t \\ p\vert M}} \Big(\frac{(1 \star f)(p^{\nu_p(M)})}{p^{\nu_p(M)}} \Big)\prod_{\substack{p \geqslant t \\ p\vert M}}\Big(\frac{ (1 \star f)(p^{\nu_p(M)})}{p^{\nu_p(M)}} + \frac{(e^{\frac{100 t}{p}} - 1) (1 \star f)(p^{\nu_p(M) - 1})}{p^{\nu_p(M)}}\Big)\\
&\leqslant Me^{-100ct} \prod\limits_{\substack{p \geqslant t \\ p \vert M}} \Big( 1 + \frac{ e^{\frac{100 t}{p}} - 1 }{p}\Big).
\end{align*}
Since \[\prod_{\substack{p \geqslant t}}\Big(1 + \frac{e^{\frac{100t}{p}} - 1}{p}\Big) \leqslant \prod_{\substack{p \geqslant t}}\Big(1 + O\Big(\frac{t}{p^2}\Big)\Big)\leqslant e^{ O(\sum_{p \geqslant t} \frac{t}{p^2})}\ll 1,\]
the lemma is proved. 
\end{proof}

\section{Proof of Proposition \ref{Prop:resolvingminimalcounterexample}}
Let $\psi,\theta,\varepsilon,C,t,f,g,\mathcal{E}^{\prime}$ be as in the statement of Proposition \ref{Prop:resolvingminimalcounterexample}, and $N$ be the natural number from property (3) of Proposition \ref{Prop:structureofminimalcounterexample} that $\mathcal{E}^{\prime}$ satisfies. We will initially prove that 
\begin{equation}
\label{eq:halfbound}
\mu_{\psi,\theta}^{f,g}\left(\mathcal{E}^{\prime}\right) \ll \left(\mu_{\psi}^f(V^{\prime}) \mu_{\theta}^g(W^{\prime}) e^{-25 Ct}\right)^{\frac{1}{2}},
\end{equation}
where $V^{\prime}: = \mathcal{E}^{\prime} |_V$ and  $W^{\prime}: = \mathcal{E}^{\prime} |_W$.
This is nearly in contradiction to property (1) of Proposition \ref{Prop:structureofminimalcounterexample} , except that the exponent is $\frac{1}{2}$ rather than $\frac{1}{2} + \varepsilon$. There is a short final argument to deal with this point. 

\begin{proof}[Proof of \eqref{eq:halfbound}]
 For each $(v,w) \in V^{\prime} \times W^{\prime}$, we define
\begin{align*}
v^-&:= \prod_{p : \,\nu_p(v/N) = -1} p,\, \qquad v^+\,\,:=  \prod_{p : \,\nu_p(v/N) = +1} p,\\
w^-&:= \prod_{p : \,\nu_p(w/N) = -1} p,\, \qquad  w^+:=  \prod_{p : \,\nu_p(w/N) = +1} p.
\end{align*}
Since $\mathcal{E}^{\prime}$ satisfies property (3), if $(v,w) \in \mathcal{E}^{\prime}$ we have that all four of $v^-, v^+, w^-, w^+$ are coprime, and \[v = N\frac{v^+}{v^-}, \qquad w = N \frac{w^+}{w^-}.\] Therefore, as $(v,w) \in \mathcal{E}^{t,C}_{\psi,\theta}$, \[ \psi(v) \leqslant \frac{ \gcd(v,w)}{w} = \frac{1}{v^{-} w^{+}}, \qquad \theta(w) \leqslant \frac{\gcd(v,w)}{v} = \frac{1}{v^{+} w^{-}}.\] Let $w_0 \in W^{\prime}$ maximise $w_0^+$, and let $v_0(w) \in \Gamma_{\mathcal{E}^{\prime}}(w)$ maximise $v_0^+(w)$ over $\Gamma_{\mathcal{E}^{\prime}}(w)$ \ed{\!\!, that is, $w^+ \leq w_0^+$ for all $w \in W'$, and $v^+ \leq v_0^+(w)$ for all $v \in \Gamma_{\mathcal{E}'}(w)$}. By applying property (2) twice, we have 
\begin{align}
\label{eq:qualitybounding}
\frac{\mu_{\psi,\theta}^{f,g}(\mathcal{E}^{\prime})}{(\mu_{\psi}^f(V^{\prime}) \mu_{\theta}^g(W^{\prime}))^{\frac{1}{2}}} \leqslant (q^\prime)^{\frac{1}{2}} \Big( \frac{\mu_{\psi,\theta}^{f,g}(\mathcal{E}^{\prime})}{\mu_{\psi}^f(V^{\prime})}\Big)^{\frac{1}{2}} \mu_{\psi}^f(\Gamma_{\mathcal{E}^{\prime}}(w_0))^{\frac{1}{2}} \leqslant q^\prime \Big( \sum\limits_{v \in \Gamma_{\mathcal{E}^{\prime}}(w_0)} \mu_{\psi}^f(v) \mu_{\theta}^g(\Gamma_{\mathcal{E}^{\prime}}(v))\Big)^{\frac{1}{2}}.
\end{align}
Using the bounds 
\begin{equation}
\label{eq:upper_bound_psi}
\psi(v) \leqslant  \frac{1}{v^{-} w_0^{+}}, \qquad \theta(w) \leqslant \frac{1}{v^{+}_0(w) w^{-}}
\end{equation} for  $(v,w) \in \mathcal{E}^{\prime}$ with $(v,w_0) \in \mathcal{E}^{\prime}$, we obtain 
\begin{align*}
 \sum\limits_{v \in \Gamma_{\mathcal{E}^{\prime}}(w_0)} \mu_{\psi}^f(v) \mu_{\theta}^g(\Gamma_{\mathcal{E}^{\prime}}(v))& = \sum\limits_{v \in \Gamma_{\mathcal{E}^{\prime}}(w_0)} \frac{f(v) \psi(v)}{v} \sum_{w \in \Gamma_{\mathcal{E}^{\prime}}(v)} \frac{ g(w) \theta(w)}{w} \nonumber\\
&\leqslant \frac{1}{w_0^+} \sum_{w \in W^{\prime}} \frac{ g(w)}{w w^-} \cdot \frac{1}{v_0^+(w)} \sum\limits_{v \in \Gamma_{\mathcal{E}^{\prime}}(w)} \frac{f(v)}{v v^-}.
\end{align*}

Since this is a sum over all pairs $(v,w) \in \mathcal{E}^{\prime} \subset \mathcal{E}^{t,C}_{\psi,\theta}$, and $\frac{vw}{\gcd(v,w)^2} = v^- v^+ w^- w^+$, for every such pair we know that \[ \sum\limits_{\substack{ p \geqslant t \\ p \vert v^- v^+ w^- w^+}} \frac{1}{p} \geqslant C.\] Hence either $\sum\limits_{\substack{p \geqslant t: \, p \vert v^-}} \frac{1}{p} \geqslant \frac{C}{4}$ or similarly with $p\vert v^{+}$, $p\vert w^{-}$ or $p \vert w^+$. Therefore \[\sum\limits_{v \in \Gamma_{\mathcal{E}^{\prime}}(w_0)} \mu_{\psi}^f(v) \mu_{\theta}^g(\Gamma_{\mathcal{E}^{\prime}}(v))  \leqslant S_1 + S_2 + S_3 + S_4,\] where 
\begin{equation}
\label{eq:S1definition}
S_1 := \frac{1}{w_0^+} \sum_{w \in W^{\prime}} \frac{ g(w)}{w w^-} \cdot \frac{1}{v_0^+(w)} \sum\limits_{\substack{v \in \Gamma_{\mathcal{E}^{\prime}} (w)\\ \sum\limits_{\substack{p \geqslant t \\ p \vert v^-}} \frac{1}{p} \geqslant \frac{C}{4}}} \frac{f(v)}{v v^-},
\end{equation}
and $S_2, S_3, S_4$ are similar expressions with the anatomy condition placed on $p\vert v^+$, $p\vert w^-$, and $p\vert w^+$ respectively. \\

\textbf{Bounding $S_1$}: We bound the inner sum from \eqref{eq:S1definition}. Since $v = v^+ \cdot \frac{N}{v^-}$ and $v^+ \leqslant v_0^+(w)$, we can reparametrise $v$ in terms of variables $v^+$ and $v^-$ and obtain
\begin{align}
\label{eq:vplusvminus}
\sum\limits_{\substack{v \in \Gamma_{\mathcal{E}^{\prime}} (w)\\ \sum\limits_{\substack{p \geqslant t \\ p \vert v^-}} \frac{1}{p} \geqslant \frac{C}{4}}} \frac{f(v)}{v v^-} &\leqslant \sum\limits_{\substack{v^+ \leqslant v_0^+(w)}} \sum\limits_{\substack{v^- \vert N \\ \gcd(v^-, v^+) = 1 \\ \sum\limits_{\substack{p\geqslant t \\ p \vert v^-}} \frac{1}{p} \geqslant \frac{C}{4}}} \frac{ f(v^+ \cdot \frac{N}{v^-})}{Nv^+}.
\end{align}
We factorise the $f$ term to remove the influence of $v^+$. Indeed, writing \[N|_{v^+}: = \prod_{p \vert \gcd(N,v^+)} p^{\nu_p(N)},\] we have 
\begin{align}\label{eq:bound_F_N}
\frac{ f(v^+ \cdot \frac{N}{v^-})}{Nv^+} &= \frac{ f\left(v^+ \cdot (N|_{v^+}) \cdot \frac{N}{N|_{v^+}v^-}\right)}{v^+ \cdot (N|_{v^+}) \cdot \frac{N}{N|_{v^+}}}  =\frac{ f\left(v^+ \cdot (N|_{v^+})\right)}{v^+ \cdot (N|_{v^+})}\cdot \frac{f\left(\frac{N}{N|_{v^+}v^-}\right)}{\frac{N}{N|_{v^+}}} \leqslant \frac{f\left(\frac{N}{N|_{v^+}v^-}\right)}{\frac{N}{N|_{v^+}}}.
\end{align}
For notational convenience, in the following we write $N^+$ for $\frac{N}{N|_{v^+}}$. In this notation, the right-hand side of \eqref{eq:vplusvminus} is 
\begin{align}
\label{eq:similar_to_DS}
\leqslant \sum\limits_{v^+ \leqslant v_0^+(w)} \Big( \frac{1}{N^+} \sum\limits_{\substack{v^- \vert N^+ \\ \sum\limits_{\substack{p\geqslant t \\ p \vert v^-}} \frac{1}{p} \geqslant \frac{C}{4}}} f(\tfrac{N^+}{v^-})\Big)\leqslant  \sum\limits_{\substack{v^+ \leqslant v_0^+(w)}}  \Big( \frac{1}{N^+} \sum\limits_{\substack{mn = N^+ \\ \sum\limits_{\substack{p \geqslant t \\ p \vert m}} \frac{1}{p} \geqslant \frac{C}{4}}} f(n)\Big) \ll v_0^+(w) e^{-25Ct},
\end{align}
by Lemma \ref{Lemma:divisoranatomy}. Therefore, the right-hand side of expression \eqref{eq:S1definition} is \[ \ll \frac{e^{-25Ct}}{w_0^+} \sum_{w \in W^{\prime}} \frac{g(w)}{w w^-} \ll \frac{e^{-25Ct}}{w_0^+} \sum\limits_{\substack{w^+ \leqslant w_0^+}} \sum\limits_{\substack{w^- \vert N \\ \gcd(w^-, w^+) = 1}} \frac{g(w^+ \cdot \frac{N}{w^-})}{N w^+} \ll e^{-25Ct},\] where the final step followed from the method we used to bound expression \eqref{eq:vplusvminus}, using $(1 \star g)(n) \leqslant n$ instead of Lemma \ref{Lemma:divisoranatomy}. Hence $S_1 \ll e^{-25Ct}$. \\

\textbf{Bounding $S_2$}. We have 
\begin{equation*}
S_2: = \frac{1}{w_0^+} \sum\limits_{w \in W^{\prime}} \frac{g(w)}{w w^-} \cdot \frac{1}{v_0^+(w)} \sum\limits_{\substack{ v \in \Gamma_{\mathcal{E}^{\prime}}(w) \\ \sum\limits_{\substack{p \geqslant t \\ p\vert v^+}} \frac{1}{p} \geqslant \frac{C}{4}}} \frac{f(v)}{v v^-}.
\end{equation*}
We bound the inner sum. Once again using $v = v^+ \cdot \frac{N}{v^-}$ and $v^+ \leqslant v_0^+(w)$, we can reparametrise $v$ in terms of variables $v^+$ and $v^-$. {{We use (as in the $S_1$ case) the estimates \eqref{eq:vplusvminus} and \eqref{eq:bound_F_N} to derive

\begin{align*}\sum\limits_{\substack{ v \in \Gamma_{\mathcal{E}^{\prime}}(w) \\ \sum\limits_{\substack{p \geqslant t \\ p\vert v^+}} \frac{1}{p} \geqslant \frac{C}{4}}} \frac{f(v)}{v v^-} \leqslant \sum\limits_{\substack{v^+ \leqslant v_0^+(w) \\ \sum\limits_{\substack{p \geqslant t \\ p\vert v^+}} \frac{1}{p} \geqslant \frac{C}{4}}} \sum\limits_{\substack{v^- \vert N \\ \gcd(v^-, v^+) = 1}} \frac{ f(v^+ \cdot \frac{N}{v^-})}{N v^+} &\leqslant \sum\limits_{\substack{v^+ \leqslant v_0^+(w) \\ \sum\limits_{\substack{p \geqslant t \\ p\vert v^+}} \frac{1}{p} \geqslant \frac{C}{4}}}\Big( \frac{1}{N^+}\sum\limits_{v^- \vert N^+} f( \tfrac{N^+}{v^{-}})\Big)
\\&\leqslant \sum\limits_{\substack{v^+ \leqslant v_0^+(w) \\ \sum\limits_{\substack{p \geqslant t \\ p\vert v^+}} \frac{1}{p} \geqslant \frac{C}{4}}}1,\end{align*}
where we applied $(1 \star f)(n) \leq n$ in the last inequality.
By an application of Lemma \ref{Lemma:unweightedanatomy}, we obtain
\[\sum\limits_{\substack{v^+ \leqslant v_0^+(w) \\ \sum\limits_{\substack{p \geqslant t \\ p\vert v^+}} \frac{1}{p} \geqslant \frac{C}{4}}}1 \ll v_0^+(w) e^{-25Ct}.\]}}

\textbf{Bounding $S_3$, $S_4$}. The bound $S_3 \ll e^{-25 Ct}$ follows by an analogous argument to the one for bounding $S_1$, with the non-trivial application of Lemma \ref{Lemma:divisoranatomy} now made on the $w$ sum instead of the $v$ sum; the bound $S_4 \ll e^{-25 Ct}$ follows by an analogous argument to the one for bounding $S_2$, with the non-trivial application of Lemma \ref{Lemma:unweightedanatomy} made on the $w$ sum instead of the $v$ sum. \\

Therefore $S_1 + S_2 + S_3 + S_4 \ll e^{-25 Ct}$. Substituting this bound into \eqref{eq:qualitybounding} and using $q^\prime < 2$, we have \[ \frac{\mu_{\psi,\theta}^{f,g}(\mathcal{E}^{\prime})}{\left(\mu_{\psi}^f(V^{\prime}) \mu_{\theta}^g(W^{\prime})\right)^{\frac{1}{2}}} \ll e^{-\frac{25}{2} Ct}\] resolving equation \eqref{eq:halfbound}.  
\end{proof}

\begin{proof}[Proof of Proposition \ref{Prop:resolvingminimalcounterexample}]
Let $\mathcal{E}^{\prime}$ be as in the statement. We will show that, provided $p_0(\varepsilon)$ is sufficiently large, 
\begin{equation}
\label{eq:strongerstructurebound}
\mu_{\psi,\theta}^{f,g}(\mathcal{E}^{\prime}) \leqslant \left(\mu_{\psi}^f(V^{\prime})\mu_{\theta}^g(W^{\prime}) e^{-Ct}\right)^{\frac{1}{q^\prime}}.
\end{equation} This will imply that $\mathcal{E}^{\prime}$ does not satisfy property (1) of Proposition \ref{Prop:structureofminimalcounterexample}, and complete the proof. 

To deal with a trivial case, if $C<0$ then we may replace this value with $C=0$ whilst still preserving $\mathcal{E}^{\prime} \subset \mathcal{E}^{\psi,\theta}_{t,C}$ and strengthening the conclusion \eqref{eq:strongerstructurebound}. So we may suppose that $C \geqslant 0$. We may also suppose that $p_0(\varepsilon)$ (and therefore also $P_{\psi,\theta}(\varepsilon)$) is sufficiently large. Now suppose that \begin{equation}
\label{eq:trivialcase}
\mu_{\psi}^f(V^{\prime}) \mu_{\theta}^g(W^{\prime}) \leqslant \frac{1000^{q P_{\psi,\theta}(\varepsilon)}}{2^q} e^{-\frac{Ct q}{q^\prime}} .
\end{equation} Then trivially we have \[ \mu_{\psi,\theta}^{f,g}(\mathcal{E}^{\prime}) \leqslant  \mu_{\psi}^f(V^{\prime}) \mu_{\theta}^g(W^{\prime}) \leqslant \frac{1}{2} \cdot 1000^{P_{\psi,\theta}(\varepsilon)}  \left(\mu_{\psi}^f(V^{\prime}) \mu_{\theta}^g(W^{\prime}) e^{-Ct}\right)^{\frac{1}{q^\prime}},\] contradicting property (1) of Proposition \ref{Prop:structureofminimalcounterexample} and resolving Proposition \ref{Prop:resolvingminimalcounterexample}. We may therefore assume that \eqref{eq:trivialcase} fails. 

Using \eqref{eq:halfbound}, for some absolute constant $K$ we have
\begin{align*}
\mu_{\psi,\theta}^{f,g}(\mathcal{E}^{\prime}) &\leqslant K\left(\mu_{\psi}^f(V^{\prime}) \mu_{\theta}^g(W^{\prime}) e^{-25Ct}\right)^{\frac{1}{2}} \\
& \leqslant K\left(\mu_{\psi}^f(V^{\prime}) \mu_{\theta}^g(W^{\prime}) e^{-25Ct}\right)^{\frac{1}{2}} \cdot \Big( \frac{ 2^q e^{\frac{Ct q}{q^\prime}}}{1000^{q P_{\psi,\theta}(\varepsilon)}} \cdot \mu_{\psi}^f(V^{\prime}) \mu_{\theta}^g(W^{\prime})\Big)^{\varepsilon} \\
& \leqslant \left( \mu_{\psi}^f(V^{\prime}) \mu_{\theta}^g(W^{\prime})\right)^{\frac{1}{2} + \varepsilon} e^{-Ct(\frac{25}{2} - \varepsilon \frac{q}{q^\prime})}\\
&\leqslant \left( \mu_{\psi}^f(V^{\prime}) \mu_{\theta}^g(W^{\prime}) e^{-Ct}\right)^{\frac{1}{2} + \varepsilon}
\end{align*}
since $\varepsilon \leqslant \frac{2}{5}$, by a short computation and $P_{\psi,\theta}$ can be chosen sufficiently large. This proves \eqref{eq:strongerstructurebound}, and thus completes the whole argument. 
\end{proof}

\section{The quantitative application}
\label{Section:quantitativeapplication}
In this section, we deduce Corollary \ref{Corollary:quantitativemetricapproximations} on quantitative metric Diophantine approximation. Along the way we will derive improved (and generalised) versions of  \cite[Theorem 2, Propositions 6 and 7]{ABH23}. 
This is the second moment estimate we will prove, which improves upon  \cite[Theorem 2]{ABH23}:
\begin{Corollary}[Improved second moment bound]
\label{Corollary:second_moment_thm}
For any $\psi:\mathbb{N} \to [0,1/2]$, let $\mathcal{A}_n, \Psi(N)$ be defined as in \eqref{eq:Aq} and \eqref{eq:def_Psi}. Then for any $\varepsilon > 0$, we have
    \[ \sum_{n,m=1}^N \lambda (\mathcal{A}_n\cap \mathcal{A}_m) = \Psi(N)^2\ed{\left(1 + O_{\varepsilon}\Bigg(\exp\bigg(-\left(\log\Psi(N)\right)^{1/2 -\varepsilon}\bigg)\Bigg)\right)}.\]
\end{Corollary}
\noindent

\begin{proof}[Proof of Corollary \ref{Corollary:quantitativemetricapproximations}, assuming Corollary \ref{Corollary:second_moment_thm}]
We can assume that $\Psi(N) \to \infty$ since otherwise, the statement holds trivially.
    We follow a standard argument, adapting \cite[Proof of Theorem 1 assuming Theorem 2]{ABH23}.
Note that $\int_{0}^1 S(N,\alpha) \,\mathrm{d}\alpha = \Psi(N)$, and by Corollary \ref{Corollary:second_moment_thm} we have for any $\rho > 0$,
\begin{eqnarray} \label{vari}
\int_0^1 \left(S(N,\alpha) - \Psi(N) \right)^2 \mathrm{d}\alpha
& = & \sum_{n,m \leqslant N} \lambda( \mathcal{A}_n\cap \mathcal{A}_m) - \Psi(N)^2 = O_{\rho} \left(\frac{\Psi(N)^2}{F_{\rho}(\Psi(N))} \right),
\end{eqnarray}
where $F_{\rho}(x) := \exp\left(\log(x)^{1/2-\rho}\right)$. 
Now let $0 < \delta < 1/100$ be arbitrary. We define 
$$
N_k = \min   \{N \in \mathbb{N}:~\Psi(N) \geqslant \exp((\log k)^{c(\delta)})  \}, \qquad k \geqslant 1,
$$
where $c(\delta) := \left(1/2-\delta\right)^{-1} > 1$.
Furthermore, let
$$
\mathcal{B}_k = \left\{ \alpha \in [0,1]:~\left| S(N_k,\alpha) - \Psi(N_k) \right| \geqslant \frac{\Psi(N_k)}{F_{\delta}(\Psi(N_k))} \right\}.
$$
By Chebyshev's inequality and applying \eqref{vari} with $\rho = \delta/3$, we have
$$
\lambda\left( \mathcal{B}_k \right) \ll_{\delta} \frac{\left(F_{\delta}(\Psi(N_k))\right)^2}{F_{\delta/3}(\Psi(N_k))} \ll_{\delta}
\left(F_{\delta/2}\left(\Psi(N_k)\right)\right)^{-1}  \ll_{\delta} k^{-2}.
$$
This implies $\sum_{k=1}^\infty \lambda(\mathcal{B}_k) < \infty$, and thus the first Borel--Cantelli Lemma implies that almost all $\alpha$ are contained in at most finitely many sets $\mathcal{B}_k$. Thus for almost all $\alpha$, there exists $k_0(\alpha)$ such that for all $k \geqslant k_0(\alpha)$,
\begin{equation}\label{eq:limit_on_subsequence}
\left| S(N_k,\alpha) - \Psi(N_k) \right| \leqslant \frac{\Psi(N_k)}{F_{\delta}(\Psi(N_k))}.
\end{equation}
Since $\psi \leqslant 1/2$ by assumption, we have $\Psi(N_{k+1}) \leqslant \exp\left(\left(\log (k+1)\right)^{c\left(\delta\right)}\right) + 1/2$, and so

\[\begin{split}
\frac{\Psi(N_{k+1})}{\Psi(N_k)} 
&\leqslant \exp\left((\log (k+1))^{c(\delta)} - (\log k)^{c(\delta)}\right) + O(1/\Psi(N_k))\\
&\leqslant \exp\left(O_{\delta}\Big(\frac{(\log k)^{c(\delta)}}{k\log k}\Big)\right) + O\left(1/\Psi(N_k)\right)
\\&\leqslant 1 + O_{\delta}\Big(\frac{(\log k)^{c(\delta)-1}}{k}\Big)
\\&\leqslant 1 + O_{\delta}\Big(\frac{1}{F_{2\delta}(\Psi(N_{k+1}))}\Big).
\end{split}
\]
Here we used $\log(1 + x) \leqslant x$, $e^x = 1 + O(x)$ when $x \in (0,1]$, binomial approximation and $c(\delta) > 1$.

Therefore, for any $N$ with $N_k \leqslant N \leqslant N_{k+1}$, using \eqref{eq:limit_on_subsequence}, we have
\begin{align*}
S(N,\alpha) \leqslant 
S(N_{k+1}, \alpha )
&\leqslant \Psi(N_{k+1}) + O\left(\frac{\Psi(N_{k+1})}{F_{2\delta}(\Psi(N_{k+1}))}\right)
\\&=\Psi(N) + O\left(\frac{\Psi(N_{k+1})}{F_{2\delta}(\Psi(N_{k+1}))}\right)\\
&= \Psi(N) + O\left(\frac{\Psi(N)}{F_{2\delta}(\Psi(N))}\right).
\end{align*}
By an analogous argument we obtain a matching lower bound from the inequality $S(N,\alpha) \geqslant S(N_k,\alpha)$ . Finally we obtain that for all $N \geqslant N_0(\alpha)$,
$$
\left| S(N,\alpha)- \Psi(N) \right| = O \left(  \frac{\Psi(N)}{F_{2\delta}(\Psi(N))}\right).
$$
By choosing $\varepsilon = \frac{\delta}{2}$, this proves Corollary \ref{Corollary:quantitativemetricapproximations}. \end{proof}

It remains to prove Corollary \ref{Corollary:second_moment_thm}. To do so, we record the following consequences of Theorem \ref{Theorem:maintheorem}. Here and throughout, $\psi:[1,N]\cap \mathbb{N} \to [0,1/2]$ is the finitely supported function from Corollary \ref{Corollary:second_moment_thm}, and $D(n,m) = D_{\psi,\psi}(n,m) = \max(m\psi(n),m\psi(n))/\gcd(n,m)$ as before. For the sake of readability, we define \begin{equation*}
L_x (n,m) := \sum_{\substack{p \mid \frac{nm}{\gcd(n,m)^2}, \\ p \geqslant x}} \frac{1}{p}.
\end{equation*}

\begin{Proposition}[GCD condition bound]
\label{prop_secondmoment1}
    For any $s,\varepsilon > 0$, let 
    \[\mathcal{E}_s := 
    \left\{(n,m): 1 \leqslant n,m\leqslant N, D(n,m) \leqslant \frac{\Psi(N)}{s}\right\}.
    \]
    Then we have
    \[\sum_{(n,m) \in \mathcal{E}_s} \frac{\psi(q)\varphi(q)}{q}\frac{\psi(r)\varphi(r)}{r}
    \ll_{\varepsilon} \frac{\Psi(N)^2}{s^{1-2\varepsilon}}.
    \]
\end{Proposition}
\noindent This improves over \cite[Proposition 6]{ABH23}, which had a power of $s^{1/5}$ in place of $s^{1 - 2\varepsilon}$. This followed from the fact that $\varepsilon = 2/5$ corresponds to the approach considered in Koukoulopoulos--Maynard \cite{KM19}. We remark that this improvement is not vital to the application to Corollary \ref{Corollary:second_moment_thm}, however. 
\begin{proof}
    Apply Theorem \ref{Theorem:maintheorem} with $f=g=\varphi$, $t=1$, $C=0$, and functions $\widetilde{\psi}(q) = \widetilde{\theta}(q) = \frac{\psi(q)s}{\Psi(N)}$. Then $ D_{\widetilde{\psi}, \widetilde{\theta}}(n,m) =  \frac{D(n,m) s}{\Psi(N)} \leqslant 1$.
\end{proof}

\begin{Proposition}
\label{prop_secondmoment2}
    Let $s,t,A \geqslant 0$ fixed, let $\eta \in (0,1/2)$ and let
    \[
    \mathcal{E}_{s,t,A} := 
    \left\{(n,m): 1 \leqslant n,m\leqslant N: D(n,m) \leqslant s \Psi(N), L_{t}(n,m) \geqslant \frac{1}{A}\right\}.
    \]
    Then we have
    \[\sum_{(n,m) \in \mathcal{E}_{s,t,A}} \frac{\psi(q)\varphi(q)}{q}\frac{\psi(r)\varphi(r)}{r}
    \ll_{\eta} \Psi(N)^2s^{\frac{1}{2} + \eta}e^{-(1 - \eta)t/A}.
    \]
\end{Proposition}

\noindent 
\begin{proof}
    Apply Theorem \ref{Theorem:maintheorem} with $\widetilde{\psi}(q) = \widetilde{\theta}(q) = \frac{\psi(q)}{s\Psi(N)}$, $C = 1/A, \varepsilon = \frac{1}{2} - \eta$. Technically, we only proved Theorem \ref{Theorem:maintheorem} when $\varepsilon \in (0,\frac{2}{5}]$, but we remarked that the method of proof proceeds unaltered if $\varepsilon \in (0,\frac{1}{2})$, at the cost of increasing $1000$ to a larger constant. We remark that in order to prove Corollary \ref{Corollary:second_moment_thm}, we do not need the full strength of this result; in fact, every $\eta > 0$ suffices.
\end{proof}

We note that without the anatomic condition on $L_t$ for $t \leqslant A$, when $s \geqslant 1$ the bound above is worse than the trivial estimate. However, in the application below, the anatomic condition that determines $t,A$ depends on $s$ and is of (much) faster decay than the polynomial growth in $s$ in the first part. In the special case of \cite[Proposition 7]{ABH23} (see \cite{ABH23} for the detailed formulation and definitions), we improve from $\frac{\Psi(N)^2}{F(t)^{1/2}}$ to $\frac{\Psi(N)^2}{e^{(F(t)^{3/4-\delta})}}$ (for any $\delta > 0)$, i.e.\! we improve from polynomial to exponential decay.

\begin{proof}[Proof of Corollary \ref{Corollary:second_moment_thm}]
We follow again the ideas of \cite{ABH23} which are a refined version of the proof in \cite{KM19}. In our setting, the estimates are even more delicate, and therefore we provide the full deduction. 

Fix $\varepsilon \in (0,1/100]$, set $\delta = \varepsilon/2$ and let
\begin{equation}\label{F_def}
F_{\rho}(x):=\exp \left( \left(\log \Psi(x)\right)^{1/2-\rho}\right)
\end{equation}
for $0 < \rho < \tfrac{1}{2}$. We can assume that $\Psi(N)$ is larger than any prescribed constant, since otherwise the statement follows trivially.
In the subsequent estimates, we will frequently use Mertens' Theorem in the following two forms (for the second formulation, see \cite{Rosser_Schoenfeld_1962}):

\begin{align}
    \label{Mertens_from1}
    \sum_{p \leqslant x} \frac{1}{p} &\leqslant \log \log x + O(1),\\
\exists a > 0: \sum_{x \leqslant p \leqslant y} \frac{1}{p} &= \log \log y - \log \log x + O\left(\exp\big(-a \sqrt{\log x}\big)\right).
    \label{Mertens_best_asymptotic}
\end{align}
\noindent 
Furthermore, we will use at various places the obvious implication
\begin{equation}\label{eq:L_decr}
    L_{s}(n,m) > L_{t}(n,m) \Rightarrow s < t.
\end{equation}

We partition the index set $[1,N]^2$ into the sets
\[ \begin{split} \mathcal{E}^{1} &= \left\{ (n,m) \in [1,N]^2 \, : \, m=n \right\}, \\ \mathcal{E}^{2} &= \left\{ (n,m) \in [1,N]^2 \, : \, n \neq m, \quad D(n,m) \leqslant \sqrt{\Psi(N)}, \quad L_{\Psi(N)}(n,m) \leqslant 1 \right\}, 
\\ \mathcal{E}^{3} &= \left\{ (n,m) \in [1,N]^2 \, : \, n \neq m, \quad D(n,m) \leqslant \sqrt{\Psi(N)}, \quad L_{\Psi(N)}(n,m) > 1 \right\}, 
\\ \mathcal{E}^{4} &= \left\{ (n,m) \in [1,N]^2 \, : \, n \neq m, \quad D(n,m) > \sqrt{\Psi(N)}, \quad L_{F_{\delta}(D(n,m))} (n,m) \leqslant \frac{4}{F_{2\delta}(\Psi(N))} \right\}, \\ \mathcal{E}^{5} &= \left\{ (n,m) \in [1,N]^2 \, : \, n \neq m, \quad D(n,m) > \sqrt{\Psi(N)}, \quad L_{F_{\delta}(D(n,m))} (n,m) > \frac{4}{F_{2\delta}(\Psi(N))} \right\}.\end{split} \]

The contribution of $\mathcal{E}^{1}$ is clearly negligible:
\begin{equation}\label{edgeset1sum}
\sum_{(n,m) \in \mathcal{E}^{1}} \lambda (\mathcal{A}_n\cap \mathcal{A}_m) = \sum_{n=1}^N \lambda (\mathcal{A}_n) = \Psi (N).
\end{equation}
\noindent Now we consider $\mathcal{E}^{2}$. For any $(n,m) \in \mathcal{E}^{2}$, the condition $L_{\Psi(N)}(n,m) \leqslant 1$ together with \eqref{Mertens_from1} ensures that
\[ \begin{split} \prod_{p \mid \frac{nm}{\gcd (n,m)^2}} \Big( 1 + \frac{1}{p-1} \Big) 
&\ll  \prod_{p \mid \frac{nm}{\gcd (n,m)^2}} \Big( 1 + \frac{1}{p} \Big) \\
&\leqslant \exp \Big( \sum_{p<\Psi(N)} \frac{1}{p} + \sum_{\substack{p \mid \frac{nm}{\gcd (n,m)^2}, \\ p \geqslant \Psi(N)}} \frac{1}{p} \Big) \\ &\ll \exp \left( \log \log \Psi (N) + O(1)\right) \\ &\ll \left(\log \Psi (N)\right)^2. \end{split} \]
The overlap estimate in the form of \eqref{eq:u_const_bound} shows that for any $(n,m) \in \mathcal{E}^{2}$,
\[ \lambda (\mathcal{A}_n\cap \mathcal{A}_m) \ll \lambda (\mathcal{A}_n) \lambda (\mathcal{A}_m) (\log \Psi (N))^2 . \]
Applying Proposition~\ref{prop_secondmoment1} with $s=\sqrt{\Psi(N)}$ and $\varepsilon = \tfrac{1}{4}$ leads to
\begin{equation}\label{edgeset2sum}
\sum_{(n,m) \in \mathcal{E}^{2}} \lambda (\mathcal{A}_n\cap \mathcal{A}_m) \ll \frac{\Psi(N)^2}{\Psi(N)^{1/4}}, 
\end{equation}
which is also negligible. 

Next we consider $\mathcal{E}^{3}$. 
Writing $r(j) = \exp(\exp(j))$, and given $(n,m) \in \mathcal{E}^{3}$, let $j(n,m)$ be the maximal integer $j$ such that $L_{r(j)}(n,m) >1$. This implies
$L_{r(j+1)}(n,m) <1$, and since by $(n,m) \in \mathcal{E}^3$ we have $L_{\Psi(N)} > 1$, \eqref{eq:L_decr} implies that $j(n,m) \geqslant \lfloor \log \log \Psi (N) \rfloor$. 
Hence \eqref{Mertens_from1} implies
\[ \begin{split} \prod_{p \mid \frac{nm}{\gcd (n,m)^2}} \left( 1 + \frac{1}{p-1} \right) &\ll \exp \bigg( \sum_{\substack{p \mid \frac{nm}{\gcd (n,m)^2}, \\ p < r(j+1)}} \frac{1}{p} + \sum_{\substack{p \mid \frac{nm}{\gcd (n,m)^2}, \\ p \geqslant r(j+1)}} \frac{1}{p} \bigg) \\ &\ll \exp \bigg( \log \log (r(j+1)) + O(1) \bigg) \\ &\ll \exp (2j) = (\log r(j))^2 . \end{split} \]
Thus the overlap estimate \eqref{eq:u_const_bound} gives
\[ \lambda (\mathcal{A}_n\cap \mathcal{A}_m) \ll \lambda (\mathcal{A}_n) \lambda (\mathcal{A}_m) (\log r(j))^2 , \]
and applying Proposition~\ref{prop_secondmoment2} with $s=1, t = r(j), A = 1$ leads to
\begin{equation}\label{edgeset3sum}
\begin{split} \sum_{(n,m) \in \mathcal{E}^{3}} \lambda (\mathcal{A}_n\cap \mathcal{A}_m) &= \sum_{j \geqslant \lfloor \log \log \Psi (N) \rfloor} \sum_{\substack{(n,m) \in \mathcal{E}^{3}, \\ j(n,m)=j}} \lambda (\mathcal{A}_n\cap \mathcal{A}_m) \\ &\ll \Psi(N)^2\sum_{j \geqslant \lfloor \log \log \Psi (N) \rfloor} \left(\log r(j)\right)^2 \exp(-r(j)/2) \\ 
& \leqslant  \Psi(N)^2\int_{x = \frac{\log \log \Psi (N)}{2}}^{\infty} (\log r(x))^2 \exp(-r(x)/2)\,\mathrm{d}x \\ 
&\ll \Psi(N)^2\int_{y = r(\log \log \Psi (N)/2)}^{\infty}  \frac{\log y}{y} e^{-y/2}\,\mathrm{d}y\\
&\ll \frac{\Psi(N)^2}{\Psi(N)}\end{split}
\end{equation}
by a considerable margin.\\

Now we consider $\mathcal{E}^{4}$. For any $(n,m) \in \mathcal{E}^{4}$,
\[ \prod_{\substack{p \mid \frac{nm}{\gcd (n,m)^2}, \\ p>F_{\delta}(D(n,m))}} \left( 1+\frac{1}{p-1} \right) \leqslant \exp \left( 2 L_{F_{\delta}(D(n,m))} (n,m) \right) = 1+O \left( \frac{1}{F_{2\delta}(\Psi(N))} \right) . \]
The overlap estimate \eqref{eq:overlap_optimized} with $\rho = \delta$ thus gives
\[ \lambda (\mathcal{A}_n\cap \mathcal{A}_m) \leqslant \lambda (\mathcal{A}_n) \lambda (\mathcal{A}_m) \left( 1 + O \left( \frac{1}{F_{2\delta}(\Psi(N))} \right) \right)\]
and hence,
\begin{equation}\label{edgeset4sum}
\sum_{(n,m) \in \mathcal{E}^{4}} \lambda (\mathcal{A}_n\cap \mathcal{A}_m) \leqslant \Psi (N)^2 + O \left( \frac{\Psi(N)^2}{F_{2\delta}(\Psi(N))}\right).
\end{equation}

Finally, we consider $\mathcal{E}^{5}$. 
Defining for shorter notation $H(x) := \exp(F_{5\delta}(x))$, we note that any $(n,m) \in \mathcal{E}^{5}$ has to lie in at least one of the following sets:

\[
\begin{split}
\mathcal{F}^1 := \Bigg\{&(n,m) \in [1,N]^2: 
\sqrt{\Psi(N)} \leqslant D(n,m) < H(\Psi(N)), 
\\&L_{F_{\delta}(D(n,m))}(n,m) \geqslant {\frac{4}{F_{2\delta}(\Psi(N))}}, L_{\sqrt{D(n,m)}}(n,m) \leqslant 10
\Bigg\},\\
\mathcal{F}^2 := \Big\{&(n,m) \in [1,N]^2: 
D(n,m) > H(\Psi(N)),  L_{F_{\delta}(D(n,m))}(n,m) \geqslant {\frac{4}{F_{2\delta}(\Psi(N))}}
\Big\},\\
\mathcal{F}^3 := \{&(n,m) \in [1,N]^2:  D(n,m) > \sqrt{\Psi(N)},  L_{\sqrt{D(n,m)}}(n,m) \geqslant 10
\}.
\end{split}
\]

We start to treat $\mathcal{F}^1$.
By \eqref{eq:u_const_bound} we have for every $(n,m) \in \mathcal{F}^1$ that
\[\lambda(\mathcal{A}_n\cap \mathcal{A}_m) \ll \lambda(\mathcal{A}_n)\lambda(\mathcal{A}_m).\]
Furthermore, using $D(n,m) \geqslant \sqrt{\Psi(N)}$, we have \[L_{F_{\delta}\left(\sqrt{\Psi(N)}\right)}(n,m) \geqslant L_{F_{\delta}(D(n,m))}(n,m) \geqslant {\frac{4}{F_{2\delta}(\Psi(N))}}.\] An application of Proposition \ref{prop_secondmoment2} with $s = H(\Psi(N)), t = F_{\delta}(\sqrt{\Psi(N)}),
A = {\frac{F_{2\delta}(\Psi(N))}{4}}$ leads to

\begin{equation}\label{edgesetF1sum}
    \sum_{(n,m) \in \mathcal{F}^1}\lambda(\mathcal{A}_n\cap \mathcal{A}_m) \ll \Psi(N)^2 H(\Psi(N))\exp\left(- \frac{F_{\delta}(\sqrt{\Psi(N)})}{F_{2\delta}(\Psi(N))}\right) \ll \frac{\Psi(N)^2}{\exp\left(F_{3\delta}(\Psi(N))\right)}
    \ll \frac{\Psi(N)^2}{\Psi(N)}.
\end{equation}

Next, we treat $\mathcal{F}^2$. For simpler notation, let $t(i) := \exp \left(\exp \left(\frac{i}{F_{2\delta}(\Psi(N))}\right)\right)$.
For any $(n,m) \in \mathcal{F}^2$, let $i(n,m)$ be the maximal integer $i$ such that
\[ L_{F_{\delta}(t(i))}(n,m) > \frac{2}{F_{2\delta}(\Psi(N))}.\] Note that
\[ L_{F_{\delta}( H(\Psi(N))} (n,m) \geqslant L_{F_{\delta}(D(n,m))} (n,m) > \frac{2}{F_{2\delta}(\Psi(N))}, \]
but
\begin{equation}\label{eq:L_F2}L_{F_{\delta}(t(i+1))}(n,m) < \frac{2}{F_{2\delta}(\Psi(N))}, \end{equation}
thus by \eqref{eq:L_decr} we have

\begin{align}
\label{size_of_i}i(n,m) &\geqslant \frac{1}{2}F_{2\delta}(\Psi(N))\log \log H(\Psi(N))
\geqslant F_{2\delta}\left(\Psi(N))(\log F_{6\delta}(\Psi(N))\right). \end{align}

Let $(n,m) \in \mathcal{F}^{2}$ such that $i(n,m)=i$. 
Using \eqref{eq:L_F2} and \eqref{Mertens_from1} shows that
\[ \begin{split} \prod_{p \mid \frac{nm}{\gcd (n,m)^2}} \left( 1+\frac{1}{p-1} \right) &\ll \exp \Bigg( \sum_{\substack{p \mid \frac{nm}{\gcd (n,m)^2}, \\ p < F_{\delta}(t(i+1))}} \frac{1}{p} + \sum_{\substack{p \mid \frac{nm}{\gcd (n,m)^2}, \\ p \geqslant F_{\delta}(t(i+1))}} \frac{1}{p} \Bigg) \\ &\ll \exp \left(\log \log F_{\delta}(t(i+1)) + O(1)\right) \\ &\ll \exp \left( \frac{i}{F_{2\delta}(\Psi(N))} \right). \end{split} \] The overlap estimate \eqref{eq:u_const_bound} thus gives
\[ \lambda (\mathcal{A}_n\cap \mathcal{A}_m) \ll \lambda (\mathcal{A}_n) \lambda (\mathcal{A}_m) \exp \left( \frac{i}{F_{2\delta}(\Psi(N))} \right) . \]

Note that \eqref{Mertens_best_asymptotic} gives
\[ \begin{split} \sum_{F_{\delta}(t(i)) \leqslant p \leqslant F_{\delta}(t(i+1))}\frac{1}{p} &\leqslant 
\log\log F_{\delta}(t(i+1)) - \log\log F_{\delta}(t(i)) + O\left(
\exp\left(-a \sqrt{\log F_{\delta}(t(i))}\right)\right)\\&\leqslant \frac{2}{F_{2\delta}(\Psi(N))}.\end{split} \]
Here we used \eqref{size_of_i}, the mean value theorem and that for large enough $x$, $h'(x) \leqslant 1$ where $h(x) := \log \log F_{\delta}\left(\exp\Big(\exp\big(x/F_{2\delta}(\Psi(N))\big)\Big)\right).$
Combined with \eqref{eq:L_F2}, it follows that
\[ L_{F_{\delta}(t(i))} (n,m) \leqslant \frac{2}{F_{2\delta}(\Psi(N))} + \frac{2}{F_{2\delta}(\Psi(N))} = \frac{4}{F_{2\delta}(\Psi(N))}, \]
and since $(n,m) \in \mathcal{F}^2$ implies per definition $L_{F_{\delta}(D(n,m))}(n,m) \geqslant {\frac{4}{F_{2\delta}(\Psi(N))}}$, \eqref{eq:L_decr} shows that
$D(n,m) \leqslant t(i)$. Therefore, Proposition~\ref{prop_secondmoment2} with 
$s = t(i)$, $T = F_{\delta}(t(i))$ and $A = \frac{F_{2\delta}(\Psi(N))}{4}$ gives 

\[ \begin{split} \sum_{\substack{(n,m) \in \mathcal{E}^{5}, \\ i(n,m)=i}} \lambda (\mathcal{A}_n\cap \mathcal{A}_m) &\ll \Psi(N)^2\exp \left( \frac{2i}{F_{2\delta}(\Psi(N))}\right)t(i) \exp\Big(-F_{\delta}(t(i))/F_{2\delta}(\Psi(N))\Big).
\end{split}\]
Note that by \eqref{size_of_i}, $t(i) \geqslant \Psi(N)$ and thus, the above is bounded by
\[\begin{split}&\Psi(N)^2\exp \left( \frac{2i}{F_{2\delta}(\Psi(N))}\right)t(i)\exp\Big(-F_{\delta}(t(i))/F_{2\delta}(t(i))\Big)\\ \ll\,&\Psi(N)^2\exp \left( \frac{2i}{F_{2\delta}(\Psi(N))}\right)t(i)\exp\Big(-F_{2\delta}(t(i))\Big)
\\=\, &\, \Psi(N)^2\left(\log \left(t(i)\right)\right)^2\;t(i)\exp\Big(-F_{2\delta}(t(i))\Big)
\\\ll &\, \Psi(N)^2 t(i)^2\exp\Big(-F_{3\delta}(t(i))\Big).
\end{split}
\]

Summing over all possible values of $i$ (see \eqref{size_of_i}), we obtain
\begin{equation}\label{edgesetF2sum}
\begin{split} \sum_{(n,m) \in \mathcal{F}^{2}}  \lambda (\mathcal{A}_n\cap \mathcal{A}_m) &\ll 
\Psi(N)^2\sum_{i \geqslant F_{2\delta}(\Psi(N))(\log F_{6\delta}(\Psi(N)))} t(i)^2\exp\left(-F_{3\delta}(t(i))\right) \\ &\ll 
\Psi(N)^2\int_{ F_{2\delta}(\Psi(N))(\log F_{7\delta}(\Psi(N)))}^{\infty} t(x)^2\exp\left(-F_{3\delta}(t(x))\right)\,\mathrm{d}x
\\&\ll \Psi(N)^2\int_{\exp(F_{7\delta}(\Psi(N)))}^{\infty} y \exp(-F_{3\delta}(y))\,\mathrm{d}y\\
&\ll \frac{\Psi(N)^2}{\Psi(N)} . \end{split}
\end{equation}

Thus we are left to treat $\mathcal{F}^3$. Let $r(j) = \exp(\exp(j))$, and for $(n,m) \in \mathcal{F}^3$, let $j(n,m)$ be the maximal integer $j$ such that $L_{\sqrt{r(j)}}(n,m) >5$.
Since $D(n,m) \geqslant \sqrt{\Psi(N)}$, \eqref{eq:L_decr} implies $j(n,m) \geqslant \frac{\log \log \Psi (N)}{2}$. Let $(n,m) \in \mathcal{F}^3$ with $j(n,m)=j$. By definition, $L_{\sqrt{r(j+1)}}(n,m) \leqslant 5$, hence \eqref{Mertens_from1} implies
\[ \begin{split} \prod_{p \mid \frac{nm}{\gcd (n,m)^2}} \left( 1 + \frac{1}{p-1} \right) &\ll \exp \bigg( \sum_{\substack{p \mid \frac{nm}{\gcd (n,m)^2}, \\ p < \sqrt{r(j+1)}}} \frac{1}{p} + \sum_{\substack{p \mid \frac{nm}{\gcd (n,m)^2}, \\ p \geqslant \sqrt{r(j+1)}}} \frac{1}{p} \bigg) \\ &\ll \exp \left( \log \log \sqrt{r(j+1)} + O(1) \right) \\ &\ll \exp (2j) = (\log r(j))^2 . \end{split} \]
Thus the overlap estimate gives
\[ \lambda (\mathcal{A}_n\cap \mathcal{A}_m) \ll \lambda (\mathcal{A}_n) \lambda (\mathcal{A}_m) (\log r(j))^2.\]
By \eqref{Mertens_best_asymptotic}, and another application of the mean value theorem, we obtain
\[ \begin{split} \sum_{\sqrt{r(j)} \leqslant p \leqslant \sqrt{r(j+1)}}\frac{1}{p} &\leq
\log\log r(j+1) - \log\log r(j) + O\left(
\exp\left(-a \sqrt{\log r(j)}\right)\right)\leqslant 2.\end{split}\]
Thus 
\[ L_{\sqrt{r(j)}}(n,m) \leqslant 7, \]
hence $D(n,m) \leqslant r(j)$. Applying Proposition~\ref{prop_secondmoment2} with 
$s = r(j)$, $T = \sqrt{r(j)}$ and $A = 1$ gives similarly to \eqref{edgeset3sum}
\begin{equation}\label{edgesetF3sum}
\sum_{(n,m) \in \mathcal{F}^3} \lambda(\mathcal{A}_n\cap \mathcal{A}_m) 
\ll \Psi(N).
\end{equation}

Combining \eqref{edgeset1sum} - \eqref{edgesetF1sum} and \eqref{edgesetF2sum}-\eqref{edgesetF3sum} finally shows that
\[ \sum_{n,m=1}^N \lambda (\mathcal{A}_n\cap \mathcal{A}_m) \leqslant \Psi(N)^2\left(1 + O_{\delta}\left(\exp\left(-(\log \Psi(N))^{1/2 -2\delta}\right)\right)\right).\]
    Since $\delta = \varepsilon/2$, this completes the proof.
\end{proof}

\appendix
\section{The structure of the Duffin--Schaeffer counterexample}

Here we review the structure of the counterexample to Khintchine's Theorem without monotonicity, given by Duffin and Schaeffer \cite{DS41}; we then discuss the relationship between this example and the rest of the present paper. In brief, we will argue that:
\begin{itemize}
\item Proposition \ref{Prop:structureofminimalcounterexample} shows that `any counterexample to the Duffin--Schaeffer conjecture must arithmetically resemble the Duffin--Schaeffer counterexample to Khintchine's Theorem without monotonicity';
\item Proposition \ref{Prop:resolvingminimalcounterexample} is a generalisation of arguments from \cite{DS41}, showing that any such counterexamples to Khintchine's Theorem without monotonicity fail to be counterexamples when a $\varphi(q)/q$ weight is introduced. 
\end{itemize}
Since Proposition \ref{Prop:structureofminimalcounterexample} works provided $\frac{f(n)}{n}$ and $\frac{g(n)}{n}$ are $1$-bounded multiplicative functions (in particular if $f(n) \equiv g(n) \equiv n$), a similar analysis may be possible for the setting of Khintchine's theorem directly.

We start by recalling the Khintchine set-up. Translated to set-theoretic language,
Khintchine's Theorem asks under which conditions on $\psi: \mathbb{N} \to [0,\tfrac{1}{2}]$ one has

\[\lambda\left(\limsup_{n \to \infty} E_n\right) = 1\]
where

\begin{equation}
\label{eq:Eq}
 E_n:=  E_n(\psi) := \bigcup\limits_{\substack{0 \leqslant a \leqslant n }} \Big[ \frac{a}{n} - \frac{\psi(n)}{n}, \frac{a}{n} + \frac{\psi(n)}{n}\Big] \cap [0,1].
\end{equation}
\noindent 
The condition 
$\sum_{n \in \mathbb{N}} \psi(n) = \frac{1}{2}\sum_{n \in \mathbb{N}} \lambda(E_n) = \infty$
is necessary, since otherwise an application of the first Borel--Cantelli Lemma proves that 
$\lambda\left(\limsup_{n \to \infty} E_n\right) = 0$. If $\sum_{n \in \mathbb{N}} \psi(n) = \infty$, one needs to have some approximate form of independence for the events $E_n$ in order to be able to use the second Borel--Cantelli lemma and conclude that $\lambda(\limsup_{n \to \infty} E_n) = 1$. Khintchine \cite{K24} showed that it sufficed to assume that $n\psi(n)$ was monotonically decreasing.
We remark that nowadays, it is well-known that Khintchine's Theorem is still true under the weaker assumption that $\psi$ is \ed{decreasing}. We refer the reader to Harman's monograph \cite{H98}.

However Duffin and Schaeffer \cite{DS41} observed that, following the trivial fact that a rational number can be written in various ways, these events can be highly dependent. Sometimes even the total dependence of the form $E_n \subseteq E_m, n \neq m$ takes place:
If $n\mid m$, then every fraction $a/n$ can also be written as some $b/m$ (since we are not assuming coprimality of numerator and denominator). Therefore, if $\psi(m)$ is proportionally larger than $\psi(n)$, every interval contributing to $E_n$ is already contained in $E_m$.

To take this dependency to the extreme, one may pick highly composite numbers and include only their divisors in the support of $\psi$, trying to maximize the occurrences of $E_n \subseteq E_m$ described above.
To formalize this, given $(\varepsilon_k)_{k \in \mathbb{N}}$ by $\varepsilon_k = 2^{-k} >0$, we construct finite, pairwise disjoint sets $(V_k)_{k \in \mathbb{N}} \subset \mathbb{N}$ such that $\sum_{v \in V_k} \lambda(E_v) > 1$ but
$\lambda\left(\bigcup_{v \in V_k} E_v\right) \ed{\leq} 2\varepsilon_k$. We pick $N_k = \prod_{N_0(k) \leqslant p < N_1(k)}p$, where $N_0(k) < N_1(k) \in \mathbb{N}$ with $N_0(k) = N_1(k-1)$; $N_1(k)$ is chosen later to be sufficiently large. Define $\psi_k: \mathbb{N} \to \mathbb{R}_{\geqslant 0}$ by 
\begin{equation}
\label{eq:DS_psi_k}
\psi_k(n) = \begin{cases} 
\frac{\varepsilon_k}{d} & \text{if } n = \frac{N_k}{d} \text{ for some } d \in \mathbb{N}, d \neq N_k \text{ dividing } N_k,\\
0 & \text{otherwise,}
\end{cases}
\end{equation}
and set $\psi := \sum_{k \in \mathbb{N}}\psi_k$.
In this way, we obtain (by implicitly using the above mentioned property that rational numbers can be written in various ways) that
\[\lambda\Big( \bigcup_{n \in \supp \psi_k} E_n\Big) = 
\lambda\Big(\bigcup_{\substack{n \mid N_k\\n \neq 1}}E_n\Big) = \lambda\left(E_{N_k}\right) = 2\varepsilon_k,
\] but \[ \sum\limits_{n \in \supp \psi_k} \lambda(E_n) = 
\sum_{\substack{n \mid N_k\\n \neq 1}} \lambda\left(E_n\right) = 2\varepsilon_k \sum_{\substack{d \mid N_k\\d \neq N_k}}\frac{1}{d} =
2\varepsilon_k\prod_{p \in [N_0(k),N_1(k)]} \left(1 + \frac{1}{p}\right) - \frac{2\varepsilon_k}{N_k}.\]
This can be made arbitrarily large when $N_1(k)$ is chosen sufficiently large, so in particular larger than $1$. Therefore, we have
\[\sum_{n \in \mathbb{N}} \psi(n) =\frac{1}{2}
\sum_{k \in \mathbb{N}}\sum_{\substack{n \mid N_k,\\n \neq 1}}  \lambda\left(E_n\right) = \infty,\]
 but \[\lambda\left(\limsup_{n \to \infty}E_n\right) \leqslant \limsup_{K_0 \to \infty}\sum_{k \geqslant K_0}\lambda\Big(\bigcup_{\substack{n \mid N_k,\\n \neq 1}} E_n\Big) = 0.\] So this function $\psi$ contradicts a na\"{i}ve generalisation of Khintchine's Theorem to all approximation functions. 

If we change to the coprime setting considered in the Duffin--Schaeffer conjecture, clearly the issue of several representations of rational numbers disappears. Writing \[\mathcal{A}_n: = \bigcup\limits_{\substack{0 \leqslant a \leqslant n \\ \gcd(a,n) = 1}} \Big[ \frac{a}{n} - \frac{\psi(n)}{n}, \frac{a}{n} + \frac{\psi(n)}{n}\Big] \cap [0,1]\] as in the introduction, and considering the same $\psi_k$ from the Duffin--Schaeffer construction, we see  \begin{equation}
\label{eq:DSexample}
 \ed{\sum_{\ed{\substack{n \in \supp \psi_k\\n \neq 1}}} \lambda(\mathcal{A}_n) = \sum_{\ed{\substack{n\mid N_k \\ n \neq 1 }}}\lambda(\mathcal{A}_n)} = 2\sum_{n\mid N_k}\frac{ \varphi(n) \psi(n)}{n} \leqslant 2\sum\limits_{n \mid N_k} \frac{ \varphi(n) \cdot \frac{\varepsilon_k}{N_k/n}}{n} = 2\frac{\varepsilon_k}{N_k}\sum\limits_{d \mid N_k} \varphi(d) = 2\varepsilon_k,\
\end{equation}
thus $\sum_{n \in \mathbb{N}} \lambda(\mathcal{A}_n) = 2\sum_{n \in \mathbb{N}} \frac{\psi(n)\varphi(n)}{n}\leqslant 2\sum_{k \in \mathbb{N}} \varepsilon_k < \infty$. Since $\sum_{n \in \mathbb{N}} \lambda(\mathcal{A}_n)$ is convergent, the construction $\psi$ does not lead to a counterexample to the Duffin--Schaeffer conjecture. \\

There are similarities between this phenomenon, as expressed in the calculation \eqref{eq:DSexample}, and the proof of Proposition \ref{Prop:resolvingminimalcounterexample}. To demonstrate the point, rather than the hypotheses of Proposition \ref{Prop:resolvingminimalcounterexample} suppose that a slightly stronger version of property (3) from Proposition \ref{Prop:structureofminimalcounterexample} holds: take $\psi: \mathbb{N} \to [0,1/2]$ with finite support and suppose there exists some $N$ such that for all non-equal $v,w \in \supp \psi$ and for all primes $p$ one has $\vert \nu_p(v/N) \vert  + \vert \nu_p(w/N) \vert \leqslant 1$. To consider the appropriate scaling, suppose also that for all $v,w \in \supp \psi$ one has $D_{\psi,\psi}(v,w) \leqslant \varepsilon_k$ (as opposed to $D_{\psi,\psi}(v,w) \leqslant 1$). This is similar to the hypotheses of Proposition \ref{Prop:resolvingminimalcounterexample}, at least morally. Then as in the proof of \eqref{eq:halfbound}, in particular expression \eqref{eq:upper_bound_psi}, for each $v \in \supp \psi$ there exist coprime square-free $v^-$, $v^+$ with $v^- \vert N$ such that $v = N \frac{v^+}{v^-}$ and $\psi(v) \leqslant \frac{\eps_k}{v^- v_0^+}$ (where $v_0^+ = \max\{ v^+: \,  v \in \supp \psi\}$). If $v^+ \equiv 1$ for all $v$, then in fact $v = \frac{N}{v^-}$ and $\psi(v) \leqslant \frac{\eps_k}{v^-}$, and so $\psi$ is bounded above by the Duffin--Schaeffer examples $\psi_k$ described above in \eqref{eq:DS_psi_k}. Thus immediately \eqref{eq:DSexample} implies that $\sum_{n \in \supp \psi} \lambda(\mathcal{A}_n) \leqslant 2\varepsilon_k$, and moreover this happens `for the same reasons' as the Duffin--Schaeffer example above fails to contradict their conjecture. 

The techniques from our proof of \eqref{eq:halfbound} resolve the case when $v^+$ varies. Indeed, in general we have
\begin{align*}
\frac{1}{2}\sum_{n \in  \supp \psi} \lambda(\mathcal{A}_n) = \sum_{n \in \supp \psi} \frac{\varphi(n) \psi(n)}{n} \leqslant \frac{\eps_k}{v_0^+} \sum_{v \in \supp \psi} \frac{\varphi(v)}{v v^-} &\leqslant \frac{\eps_k}{v_0^+} \sum_{\substack{ v^+,v^- \in \mathbb{N} \\ v^+ \leqslant v_0^+ \\
v^- \vert N \\ \gcd(v^+,v^-) = 1}} \frac{\varphi(v^+ \cdot \frac{N}{v^-})}{Nv^+}\\ &\leqslant \frac{\eps_k}{v_0^+} \sum_{v^+ \leqslant v_0^+} \frac{1}{N^+} \sum_{v^- \vert N^+} \varphi \big(\tfrac{N^+}{v^-}\big) \leqslant \eps_k,
\end{align*} where $N^+ = N/\gcd(N,v^+)$. This argument is similar to the one used for the bound \eqref{eq:similar_to_DS}. \\

There are further windows into the relationship between Duffin--Schaeffer counterexamples and Proposition \ref{Prop:structureofminimalcounterexample}. Indeed, there is a refinement of the Duffin--Schaeffer counterexample, in which the component functions $\psi_k$ are defined by 
\begin{equation}
\label{eq:refined_DS}
\psi_k(n) = \begin{cases} \frac{\varepsilon_k}{p} & \text{if } n = \frac{N_k}{p} \text{ for some prime } p \vert N_k,\\
0 & \text{otherwise.}\end{cases}
\end{equation} This still provides a counterexample to Khintchine's Theorem without monotonicity (since $\sum_p \frac{1}{p} = \infty$), but we observe that this example satisfies $\vert \nu_p(v/N_k)\vert + \vert \nu_p(w/N_k)\vert \leqslant 1$ for all $v \neq w$ for which $v,w \in \supp \psi_k$. So in fact Proposition \ref{Prop:structureofminimalcounterexample} is more precisely saying that any counterexample to Theorem \ref{Theorem:maintheorem} must contain within a structure similar to the refined Duffin--Schaeffer example \eqref{eq:refined_DS}. 

\section{Improved anatomy lemmas}

The exponential saving obtained in the lemmas of Section \ref{Section:anatomy} is weaker than the super-exponential decay obtained in the corresponding lemma of Koukoulopoulos--Maynard \cite[Lemma 7.3]{KM19}. One can in fact obtain such stronger decay in our settings, in a wide range of uniformity in the parameters $c$ and $t$: we discuss this below. However, it seems non-trivial to include these improvements within the proof of Proposition \ref{Prop:structureofminimalcounterexample}, since the parameter $C_3$ involved in Lemma \ref{Lemma:decayawayfromdiagonal} will no longer be $O(1)$ (but will depend on $C$ and $t$). There are extra complications, involving the ranges of parameters on which one can improve the anatomy bounds.

To obtain stronger bounds in Lemmas \ref{Lemma:unweightedanatomy} and \ref{Lemma:divisoranatomy}, one may adapt the following trick from the proof of \cite[Lemma 7.3]{KM19}. 

\begin{Lemma}
    \label{Lemma:anatomy_improvement_trick}
    Let $\varepsilon \in (0,1/2)$. Then for any $t \geqslant 1$ and $c \in \mathbb R$ with $\varepsilon c\log t$ sufficiently large,
    \begin{align*}
        \big\{ n \leqslant x: \sum_{\substack{p \geqslant t \\ p\mid n}} \frac{1}{p} \geqslant c\big\}
        \subseteq
        \big\{ n \leqslant x: \sum_{\substack{p \geqslant t ^{e^{(1-2 \varepsilon)c}}\\ p\mid n }} \frac{1}{p} \geqslant \varepsilon c\big\}.
    \end{align*}
    \end{Lemma}
\begin{proof}
    Set $T = t ^{e^{(1-2 \varepsilon)c}}$. From the assumptions and Mertens' Second Theorem (with the basic error term \cite[Theorem 3.4]{K20})
    \begin{align*}
        \sum_{t \leqslant p \leqslant T} \frac{1}{p} = (1-2 \varepsilon)c + O\Big(\frac{1}{\log t}\Big) \leqslant (1- \varepsilon) c.
    \end{align*}
    Therefore, whenever $\sum_{\substack{p\geqslant t \\ p\mid n}}\frac{1}{p} \geqslant c$ we must also have $\sum_{\substack{p\geqslant T \\ p\mid n}}\frac{1}{p} \geqslant \varepsilon c$ and the lemma follows.
\end{proof}
One could expand the range of $c$ in Lemma \ref{Lemma:anatomy_improvement_trick} using an improved error term in Mertens' Second Theorem, for example the bound recalled in \eqref{Mertens_best_asymptotic} above. However, there are other limiting factors which will constrain our choice of $c$ below to the range $c \gg (\log t)^{-1}$, so for simplicity we chose to use the basic error term in Lemma \ref{Lemma:anatomy_improvement_trick}. 
\begin{Lemma}
    \label{Lemma:unweightedanatomy_improvement}
    Let $\varepsilon \in (0,1)$. Suppose $t \geqslant e^{e}$ and $c > 0$ with $\varepsilon c \frac{\log t}{\log \log t}$ sufficiently large. Then for any real $x\geqslant 1$,
    \begin{align}
\label{eq:unweighted_improvement}
        \Big| \{ n \leqslant x: \sum_{\substack{p \geqslant t \\ p\mid n}} \frac{1}{p} \geqslant c\} \Big|
        \ll xe^{-100t ^{e^{(1- \varepsilon)c}}}.
    \end{align}
\end{Lemma}
\begin{proof}
    Take $\delta := \varepsilon/3$. Then $\delta \in (0, 1/3)$ and Lemmas \ref{Lemma:unweightedanatomy} and \ref{Lemma:anatomy_improvement_trick} imply
    \begin{align*}
        \Big| \{ n \leqslant x: \sum_{\substack{p \geqslant t \\ p\mid n }} \frac{1}{p} \geqslant c\} \Big|
        \ll xe^{-100\delta ct ^{e^{(1-2\delta)c}}}.
    \end{align*}
    Now observe that
    \begin{align*}
        \delta ct ^{e^{(1-2 \delta)c}}
        \geqslant \delta c t ^{\delta c}t ^{e^{(1-3 \delta)c}}
        \geqslant \delta c \log (t) t ^{e ^{(1-3 \delta) c}}
        \geqslant t ^{e^{(1-3 \delta)c}}
=  t ^{e^{(1-\varepsilon)c}}
    \end{align*}
    where we are using that $\varepsilon c \log t\geqslant 3\log \log t$. 
\end{proof}

A very similar argument applies to the second anatomy lemma.
\begin{Lemma}
    \label{Lemma:divisoranatomy_improvement}
Let $\varepsilon \in (0,1)$. Suppose $t \geqslant e^{e}$ and $c > 0$ with $\varepsilon c \frac{\log t}{\log \log t}$ sufficiently large. Let $f$ be a non-negative multiplicative function that satisfies \eqref{eq:conv_bound}. Then for any $M \in \mathbb{N}$,
    \begin{equation}
\label{eq:weighted_improvement}
\sum\limits_{\substack{mn = M \\ \sum\limits_{\substack{p \geqslant t \\ p\vert m}} \frac{1}{p} \geqslant c}} f(n)
    \ll Me^{-100t ^{e^{(1-\varepsilon)c}}}.
\end{equation}
\end{Lemma}
\begin{proof}
    See the proof of Lemma \ref{Lemma:unweightedanatomy_improvement}. The implicit constant does not depend on the size of $\varepsilon c \frac{\log t}{\log \log t}$, provided this quantity is initially chosen to be sufficiently large, nor on the choice of $f$. 
\end{proof}

The constant $100$ appearing in the exponents of equations  \eqref{eq:unweighted_improvement} and \eqref{eq:weighted_improvement} is arbitrary (and was chosen in Section \ref{Section:anatomy} to line up with the choice in \cite{ABH23}). One could replace $100$ with any positive constant $A$, at the expense of different implicit constants in equations \eqref{eq:unweighted_improvement} and \eqref{eq:weighted_improvement}. Making the $A$ dependency explicit, and optimising for the parameter $A$, would allow us to relax the condition $\varepsilon c \frac{\log t}{\log \log t} \gg 1$ to $\varepsilon c\log t \gg 1$. We do not pursue this technical line here, however.

Finally, we show that Lemma \ref{Lemma:unweightedanatomy_improvement} is very close to optimal. 

\begin{Lemma}
  Suppose $t \geqslant 1$ and $c, \varepsilon > 0$ with $\varepsilon c\log t$ sufficiently large. Then for any real $x \geqslant 1$,
    \begin{align*}
        \Big| \{ n \leqslant x: \sum_{\substack{p \geqslant t \\ p\mid n}} \frac{1}{p} \geqslant c\} \Big|
        \geqslant xe^{-t ^{e^{(1+\varepsilon)c}}} -1.
    \end{align*}
\end{Lemma}

\begin{proof}
 If $n_0$ satisfies $\sum_{\substack{p \geqslant t\\ p\mid n_0 }} \frac{1}{p} \geqslant c$ then the same holds for all multiples of $n_0$, and
    \begin{align*}
        \Big| \{ n \leqslant x: \sum_{\substack{p \geqslant t \\ p\mid n }} \frac{1}{p} \geqslant c\} \Big|
        \geqslant \frac{x}{n_0} - 1.
    \end{align*}
   Set $T = t ^{e^{(1+\varepsilon)c}}$, so that again by Mertens' Second Theorem 
    \begin{align*}
        \sum_{t\leqslant p\leqslant T}\frac{1}{p} = (1+ \varepsilon)c + O\Big(\frac{1}{\log t}\Big)
        \geqslant c.
    \end{align*}
    Setting $n_0 = \prod_{t\leqslant p \leqslant T}p$, then clearly $n_0$ has the desired property and $n_0 \leqslant e^{O(T)} \leqslant e^{t ^{e^{(1+2 \varepsilon)c}}}$ since $\varepsilon c \log t$ is sufficiently large. The lemma follows by rescaling $\varepsilon$.
\end{proof}

A similar idea allows us to obtain a lower bound when $c\log t$ is small. (We choose the upper-bound $1/2$ for concreteness.) 

\begin{Lemma}
    There is an absolute constant $K>0$ for which the following holds. If $t$ is sufficiently large, then for any real $x \geqslant 1$ and $c \in \mathbb R$ with $c\log t \leqslant 1/2$,

    \begin{align*}
        \Big| \{ n \leqslant x: \sum_{\substack{p \geqslant t \\ p\mid n}} \frac{1}{p} \geqslant c\} \Big|
        \geqslant xe^{-Kt} -1.
    \end{align*}

\end{Lemma}
\begin{proof}
  Take $T = 3t$. Then 
    \begin{align*}
        \sum_{t\leqslant p \leqslant T} \frac{1}{p} \geqslant \frac{\pi(3t)-\pi(t)}{3t} = \frac{(2+o(1))t/\log t}{3t} \geqslant \frac{1}{2\log t} \geqslant c.
    \end{align*}
    Therefore, we may take $n_0 = \prod_{t\leqslant p \leqslant T}p$ and the lemma follows in the same way as before.
\end{proof}

\bibliographystyle{plain}
\bibliography{duffinnew}

@article {GW21,
    AUTHOR = {Green, B. and Walker, A.},
     TITLE = {Extremal problems for {GCD}s},
   JOURNAL = {Combin. Probab. Comput.},
  FJOURNAL = {Combinatorics, Probability and Computing},
    VOLUME = {30},
      YEAR = {2021},
    NUMBER = {6},
     PAGES = {922--929},
      ISSN = {0963-5483,1469-2163},
   MRCLASS = {11B75 (11A05)},
  MRNUMBER = {4328356},
MRREVIEWER = {Mitsuo\ Kobayashi},
       DOI = {10.1017/s0963548321000092},
       URL = {https://doi.org/10.1017/s0963548321000092},
}

@article {S60,
    AUTHOR = {Schmidt, W.},
     TITLE = {A metrical theorem in diophantine approximation},
   JOURNAL = {Can. J. Math.},
  FJOURNAL = {Canadian Journal of Mathematics. Journal Canadien de
              Math\'{e}matiques},
    VOLUME = {12},
      YEAR = {1960},
     PAGES = {619--631},
      ISSN = {0008-414X,1496-4279},
   MRCLASS = {10.00},
  MRNUMBER = {118711},
MRREVIEWER = {J.\ W. S. Cassels},
       DOI = {10.4153/CJM-1960-056-0},
       URL = {https://doi.org/10.4153/CJM-1960-056-0},
}

@article{ABH23,
    AUTHOR = {Aistleitner, C. and Borda, B. and Hauke, M.},
     TITLE = {On the metric theory of approximations by reduced fractions: a
              quantitative {K}oukoulopoulos-{M}aynard theorem},
   JOURNAL = {Compos. Math.},
  FJOURNAL = {Compositio Mathematica},
    VOLUME = {159},
      YEAR = {2023},
    NUMBER = {2},
     PAGES = {207--231},
      ISSN = {0010-437X,1570-5846},
   MRCLASS = {11J83 (11A05 11J04 11K60)},
  MRNUMBER = {4546786},
MRREVIEWER = {Faustin\ Adiceam},
       DOI = {10.1112/S0010437X22007837},
       URL = {https://doi.org/10.1112/S0010437X22007837},
}

@article {R17,
    AUTHOR = {Ram\'{\i}rez, F.},
     TITLE = {Counterexamples, covering systems, and zero-one laws for
              inhomogeneous approximation},
   JOURNAL = {Int. J. Number Theory},
  FJOURNAL = {International Journal of Number Theory},
    VOLUME = {13},
      YEAR = {2017},
    NUMBER = {3},
     PAGES = {633--654},
      ISSN = {1793-0421,1793-7310},
   MRCLASS = {11J83 (11J71 11K06 11K60)},
  MRNUMBER = {3606945},
MRREVIEWER = {Dzmitry\ Badziahin},
       DOI = {10.1142/S1793042117500324},
       URL = {https://doi.org/10.1142/S1793042117500324},
}

@article{Aistleitner_2014,
title={A note on the {D}uffin-{S}chaeffer conjecture with slow divergence}, 
volume={46},
number={1},
journal={Bull. Lond. Math. Soc},
author={Aistleitner, C.},
year={2014},
pages={164–168} 
}

@article{Aistleitner_Lachmann_Munsch_Technau_Zafeiropoulos_2019, 
title={The {D}uffin-{S}chaeffer conjecture with extra divergence},
volume={356},
number={106808},
journal={Adv. Math.},
author={Aistleitner, C. and Lachmann, T. and Munsch, M. and Technau, N. and Zafeiropoulos, A.}, 
year={2019}
}

@article{Beresnevich_Harman_Haynes_Velani_2013,
title={The {D}uffin–{S}chaeffer conjecture with extra divergence {II}},
volume={275},
number={1–2}, 
journal={Math.Z.},
author={Beresnevich, V. and Harman, G. and Haynes, A. and Velani, S.},
year={2013},
pages={127–133}
}

@article{Haynes_Pollington_Velani_2012,
title={The {D}uffin–{S}chaeffer {C}onjecture with extra divergence},
volume={353},
number={2},
journal={Math. Ann.}, 
author={Haynes, A. and Pollington, A. and Velani, S.},
year={2012},
pages={259–273}
}

@article{K24,
  title={{E}inige {S}{\"a}tze {\"u}ber {K}ettenbr{\"u}che, mit {A}nwendungen auf die {T}heorie der {D}iophantischen {A}pproximationen},
  author={Khintchine, A.},
  journal={Math. Ann.},
  volume={92},
  pages={115--125},
  year={1924}
}

@article{DS41,
  title={Khintchine’s problem in metric Diophantine approximation},
  author={Duffin, R. and Schaeffer, A.},
  journal={Duke Math. J.},
  volume={8},
  number={2},
  pages={243--255},
  year={1941},
  publisher={Duke University Press}
}

@article{KM19,
  title={On the {D}uffin-{S}chaeffer conjecture},
  author={Koukoulopoulos, D. and Maynard, J.},
  journal={Ann. Math. (2)},
  volume={192},
  number={1},
  pages={251--307},
  year={2020},
  publisher={JSTOR}
}

@article{Va78,
  title={On the metric theory of Diophantine approximation},
  author={Vaaler, J.},
  journal={Pac. J. Math.},
  volume={76},
  number={2},
  pages={527--539},
  year={1978},
  publisher={Mathematical Sciences Publishers}
}

@misc{KLL25,
 author = {Koukoulopoulos, D. and Lamzouri, Y. and Lichtman, J.D.},
 title = {Erd{\H{o}}s's integer dilation approximation problem and {GCD} graphs},
 year = {2025},
 howpublished = {Preprint, {arXiv}:2502.09539 [math.{NT}] (2025)},
 url = {https://arxiv.org/abs/2502.09539},
 arXiv = {arXiv:2502.09539}
}

@article{KMY24,
 author = {Koukoulopoulos, D. and Maynard, J. and Yang, D.},
 title = {An almost sharp quantitative version of the {Duffin}-{Schaeffer} conjecture},
 fjournal = {Duke Mathematical Journal},
 journal = {Duke Math. J.},
 issn = {0012-7094},
 volume = {174},
 number = {10},
 pages = {2011--2065},
 year = {2025},
 language = {English},
 doi = {10.1215/00127094-2024-0055}
}

@article{Er70,
  title={On the distribution of the convergents of almost all real numbers},
  author={Erd\H{o}s, P.},
  journal={J. Number Theory},
  volume={2},
  number={4},
  pages={425--441},
  year={1970},
  publisher={Elsevier}
}

@article{PV90,
  title={The k-dimensional {D}uffin and {S}chaeffer conjecture},
  author={Pollington, A. and Vaughan, R.},
  journal={Mathematika},
  volume={37},
  pages={190--200},
  year={1990}
}

@article{G61,
    AUTHOR = {Gallagher, P.},
     TITLE = {Approximation by reduced fractions},
   JOURNAL = {J. Math. Soc. Japan},
  FJOURNAL = {Journal of the Mathematical Society of Japan},
    VOLUME = {13},
      YEAR = {1961},
     PAGES = {342--345},
      ISSN = {0025-5645,1881-1167},
   MRCLASS = {10.30},
  MRNUMBER = {133297},
MRREVIEWER = {J.\ W. S. Cassels},
       DOI = {10.2969/jmsj/01340342},
       URL = {https://doi.org/10.2969/jmsj/01340342},
}

@article{Rosser_Schoenfeld_1962,
title={Approximate formulas for some functions of prime numbers}, 
volume={6},
number={1},
journal={Ill. J. Math.},
author={Rosser, J. and Schoenfeld, L.},
year={1962}, 
pages={64–94}
}

@book {H98,
    AUTHOR = {Harman, G.},
     TITLE = {Metric number theory},
    SERIES = {London Mathematical Society Monographs. New Series},
    VOLUME = {18},
 PUBLISHER = {The Clarendon Press, Oxford University Press, New York},
      YEAR = {1998},
     PAGES = {xviii+297},
      ISBN = {0-19-850083-1},
   MRCLASS = {11J83 (11K38 11K50 11K55)},
  MRNUMBER = {1672558},
MRREVIEWER = {R.\ C.\ Baker},
}

@book{K20,
  title={The distribution of prime numbers},
  author={Koukoulopoulos, D.},
  year={2020},
  publisher={American Mathematical Soc.}
}

\vspace{5mm}

\end{document}